\documentclass[10pt,reqno]{amsart}
\usepackage{amsmath, hyperref}
\usepackage{amsthm}
\usepackage{amssymb}
\usepackage{amsfonts}
\usepackage{latexsym}
\usepackage{hyperref}
\usepackage{cite}
\usepackage{xcolor}

\usepackage{enumitem}
\usepackage{mathtools}
\mathtoolsset{showonlyrefs}

\usepackage{graphicx}
\graphicspath{{./GRAPHS/}}
\usepackage{subcaption}

\usepackage{tikz,ifthen}
\usetikzlibrary{intersections}
\usetikzlibrary{calc}
\usetikzlibrary{decorations.pathreplacing}
\usetikzlibrary{calc,decorations.markings}

\usetikzlibrary{positioning}
\usetikzlibrary{arrows,shapes,backgrounds,3d}
\usepackage{fix-cm}
\usetikzlibrary{decorations.pathreplacing,shapes,snakes}
\tikzstyle{box} = [rectangle, minimum width=1cm, minimum height=3cm,text centered, draw=black, fill=red!10]

\usepackage{fancyvrb}
\theoremstyle{plain}
\newtheorem{theorem}{Theorem}
\newtheorem{corollary}[theorem]{Corollary}
\newtheorem{lemma}[theorem]{Lemma}

\theoremstyle{definition}

\newtheorem{remark}[theorem]{Remark}

\renewcommand{\L}{\mathbb{L}}

\newcommand{\Q}{\mathbb{Q}}
\newcommand{\R}{\mathbb{R}}

\newcommand{\K}{\mathbb{K}}

\newcommand{\C}{\mathbb{C}}

\newcommand{\Gal}{\operatorname{Gal}}

\renewcommand{\Re}{\operatorname{Re}}
\renewcommand{\Im}{\operatorname{Im}}

\let\oldenumerate=\enumerate
	\def\enumerate{
	\oldenumerate
	\setlength{\itemsep}{5pt}
	}
\let\olditemize=\itemize
	\def\itemize{
	\olditemize
	\setlength{\itemsep}{5pt}
	}

\allowdisplaybreaks

\begin{document}
\VerbatimFootnotes
\title[Explicit estimates for Artin $L$-functions]{Explicit estimates for Artin $L$-functions: Duke's short-sum theorem and Dedekind Zeta Residues}

	\author[S.~R.~Garcia]{Stephan Ramon Garcia}
	\address{Department of Mathematics, Pomona College, 610 N. College Ave., Claremont, CA 91711} 
	\email{stephan.garcia@pomona.edu}
	\urladdr{\url{http://pages.pomona.edu/~sg064747}}
	
    \author[E.~S.~Lee]{Ethan Simpson Lee}
    \address{School of Science, UNSW Canberra at the Australian Defence Force Academy, Northcott Drive, Campbell, ACT 2612} 
    \email{ethan.s.lee@student.adfa.edu.au}
    \urladdr{\url{https://www.unsw.adfa.edu.au/our-people/mr-ethan-lee}}
	
\thanks{SRG supported by NSF Grant DMS-1800123.}

\subjclass[2010]{11R42, 11M06, 11S40}

\keywords{Number field, Artin $L$-function, Generalized Riemann Hypothesis, GRH, Dedekind zeta function, residue, prime, character}

\begin{abstract}
Under GRH, we establish a version of Duke's short-sum theorem for entire Artin $L$-functions.  This yields corresponding bounds for residues of Dedekind zeta functions.  All numerical constants in this work are explicit.
\end{abstract}

\maketitle
\section{Introduction}
In \cite{Duke}, Duke proved a remarkable ``short-sum theorem'' that relates the value of an Artin $L$-function at $s=1$ to a sum 
over an exceptionally small set of primes.  To be more specific, if $L(s,\chi)$ is an entire Artin $L$-function that satisfies the Generalized Riemann Hypothesis (GRH), Duke proved that
\begin{equation}\label{eqn:Duke_original}
\log L(1,\chi) \,\,=\!\!\!\! \sum_{p \leq (\log N)^{1/2} } \frac{\chi(p)}{p} + O(1),
\end{equation}
in which $p$ is a prime, $\chi$ has degree $d$ and conductor $N$, and the implicit constant depends only upon $d$ \cite[Prop.~5]{Duke}. 
Our main result is an explicit version of Duke's theorem
(by ``explicit'' we mean that there are no implied constants left unspecified).

\begin{theorem}\label{Theorem:Duke}
Let $L(s,\chi)$ be an entire Artin $L$-function that satisfies GRH, in which $\chi$ has degree $d$ and conductor $N$. Then
\begin{equation}\label{eq:Duke}
    \bigg| \log L(1,\chi)\,\, -\!\!\!\! \sum_{p \leq (\log N)^{1/2}} \frac{\chi(p)}{p}  \bigg|
    \,\leq\,  13.53d.
\end{equation}
\end{theorem}

The need to make every step explicit requires great precision. A wide variety of preliminary results are needed with
concrete numerical constants.  For example, we require certain gamma-function integral estimates and a version of Mertens second theorem applicable over a large range
and with all numerical constants specified.  Furthermore, several parameters must be finely tuned and optimized to reduce the upper bound in \eqref{eq:Duke}
as far as our techniques permit (all numerical computations that follow were verified independently in both Mathematica and Python).
Overall, the proof of Theorem \ref{Theorem:Duke} is much more complex than the proof of \eqref{eqn:Duke_original}.

Duke's original motivation was the construction (under GRH) of totally real number fields with certain extremal properties \cite{Duke}; see \cite{MR3179652,MR3121669} for related work also assuming the Artin or strong Artin conjectures, and
\cite{MR2276192} for progress in the unconditional case. 
Since its introduction in 2003, Duke's result has been used to study the smallest point on a diagonal cubic surface \cite[p.~192]{MR2676747}, complex moments of symmetric power $L$-functions \cite[(1.46)]{MR2035301}, upper bounds on the class number of a CM number field with a given maximal real subfield \cite[p.~938]{MR2600412},
and extreme logarithmic derivatives of Artin $L$-functions \cite[p.~583]{MR3049696}
(all under GRH). Consequently, the novel bound \eqref{eq:Duke} should lead to explicit estimates in several adjacent areas.

Theorem \ref{Theorem:Duke} yields explicit estimates for $\kappa_{\K}$,
the residue at $s=1$ of the Dedekind zeta function $\zeta_{\K}(s)$ of a number field $\K$
(the Riemann zeta function is $\zeta = \zeta_{\Q}$).
For quadratic fields, such estimates have a long history dating back to Littlewood \cite{Littlewood};
see also Chowla \cite{Chowla},
Chowla--Erd\H{o}s \cite{MR44566},
Elliott \cite{MR249374},
Granville--Soundararajan \cite{MR2024414},
and Montgomery--Vaughan \cite{MR1689558}.
Since $\kappa_{\K}$ appears in the analytic class number formula \cite[p.~161]{Lang}
and in Mertens-type theorems for number fields \cite{Rosen,Lebacque},
explicit estimates such as the following should have immediate use.

\begin{corollary}\label{Corollary:Kappa}
Assume GRH and that $\zeta_{\K}/\zeta$ is entire. 
If $\K$ is a number field of degree $n_{\K} \geq 2$ and discriminant $\Delta_{\K}$, then
\begin{equation}\label{eq:Kappa}
\frac{1}{e^{17.81 (n_{\K}-1)} \log \log | \Delta_{\K}| }
\leq \kappa_{\K} \leq  (e^{17.81}\log \log | \Delta_{\K}|)^{n_{\K}-1} .
\end{equation}
For $n_{\K}=2,3,4,5,6$, we have the following.
\begin{itemize}
\item For $n_{\K}=2$, we have
$\displaystyle\frac{1}{e^{17.81} \log \log | \Delta_{\K}|} \leq \kappa_{\K} \leq e^{17.81} \log \log | \Delta_{\K}|$.

\item For $n_{\K}=3$, we have
$\displaystyle\frac{1}{e^{17.54} \log \log | \Delta_{\K}|} \leq \kappa_{\K} \leq e^{18.87} (\log \log | \Delta_{\K}|)^2$.

\item For $n_{\K}=4$, we have
$\displaystyle\frac{1}{e^{22.52} \log \log | \Delta_{\K}|} \leq \kappa_{\K} \leq e^{24.1} (\log \log | \Delta_{\K}|)^3$.

\item For $n_{\K}=5$, we have
$\displaystyle\frac{1}{e^{26.93} \log \log | \Delta_{\K}|} \leq \kappa_{\K} \leq e^{28.2} (\log \log | \Delta_{\K}|)^4$.

\item For $n_{\K}=6$, we have
$\displaystyle\frac{1}{e^{32.24} \log \log | \Delta_{\K}|} \leq \kappa_{\K} \leq e^{33.36} (\log \log | \Delta_{\K}|)^5$.
\end{itemize}
\end{corollary}

That $\zeta_{\K}/\zeta$ is entire is known if 
$\K$ is normal or if the Galois group of its normal closure is solvable \cite{Uchida,vdW}; see \cite[Ch.~2]{MMRMVK} for more information.
In particular, this hypothesis holds for any cubic or quartic number field. 
The upper and lower bounds in \eqref{eq:Kappa} have the expected order of magnitude (under GRH) with respect to $|\Delta_{\K}|$, although the dependence on $n_{\K}$ can be improved if one uses inexplicit constants; see \cite[(1.1)]{ChoKim}.
The novel contribution in Corollary \ref{Corollary:Kappa} is the explicit dependence upon the degree and discriminant;
there are no implied constants left unspecified.

For $n_{\K}\geq 3$, the best known explicit unconditional bounds are
\begin{equation}\label{eq:UncondKappa}
    \frac{0.0014480}{n_{\K} g(n_{\K}){|\Delta_{\K}|}^{1/n_{\K}}}
    \,\,<\,\, \kappa_{\K} \,\,\leq \,\,
    \left(\frac{e\log |\Delta_{\K}| }{2(n_{\K} - 1)}\right)^{n_{\K} - 1},
\end{equation}
in which $g(n_{\K})=1$ if $\K$ has a normal tower over $\Q$ and $g(n_{\K}) = n_{\K}!$ otherwise. 
The lower bound follows from an analysis of Stark's paper \cite{StarkBS}, although his language is ambiguous;
see \cite[Rem.~13]{GarciaLeeMertens}. 
The upper bound is due to Louboutin \cite[Thm.~1]{Louboutin00}.
For entire $\zeta_{\K}/\zeta$, Louboutin refined the upper bound in \eqref{eq:UncondKappa} 
\cite[Thm.~1]{Louboutin11} and provided refinements under other
assumptions (e.g., $\K$ is totally imaginary) and depending upon the location of a possible real zero of $\zeta_{\K}$  \cite{Louboutin03, Louboutin05};  see also Ramar\'{e} \cite[Cor.~1]{RamL(1chi)_2}.

This paper is organized as follows.
The proof of Theorem \ref{Theorem:Duke}, which occupies the bulk of this paper, is in Section \ref{Section:ProofDuke}.
The proof of Corollary \ref{Corollary:Kappa} is in Section \ref{Section:ProofKappa}.

\subsection*{Acknowledgments}
Special thanks to the anonymous referee for a careful reading of this paper.
Thanks to Andrew Booker, Peter Cho, Michaela Cully-Hugill, Bill Duke, Eduardo Friedman, Edray Goins, Ken Ribet, 
Aleksander Simoni\v{c}, Valeriia Starichkova, and Tim Trudgian for helpful conversations.  


\section{Proof of Theorem \ref{Theorem:Duke}}\label{Section:ProofDuke}

We assume the Generalized Riemann Hypothesis (GRH) throughout what follows.
The proof of Theorem \ref{Theorem:Duke} require a series of lemmas, which are spread out over several subsections.
Section \ref{Subsection:Mertens} presents Mertens' second theorem with an explicit error term and a large range of applicability,
Section \ref{Subsection:DukeEstimate} involves general bounds for entire Artin $L$-functions, and
Section \ref{Subsection:DukeLogDerivative} deduces estimates for their logarithmic derivatives.
A technical integral estimate involving the gamma function appears in
Section \ref{Subsection:Gamma}.
In Section \ref{Subsection:Approx}, we approximate the logarithmic derivative of an entire Artin $L$-function using a sum over primes.
Section \ref{Subsection:Weighted} builds upon the previous material and considers a particular exponentially weighted sum, which is refined further in
Section \ref{Subsection:Exponential}. The proof of Theorem \ref{Theorem:Duke} wraps up in Section \ref{Subsection:DukeComplete}, 
where we optimize several numerical parameters.  Appendix \ref{Section:Constants} contains a convenient, cross-referenced summary of the constants
and functions that arise throughout the proof.

\subsection{Explicit Mertens' theorems}\label{Subsection:Mertens}

We require an explicit version of Mertens' second theorem
under the Riemann Hypothesis (which is implied by GRH).
We need a convenient estimate valid for $x \geq (\log 3)^{1/2} \approx 1.04815$
(this number arises because $3$ is the smallest possible value of the conductor of an entire Artin $L$-function;
see the comments at the beginning of Subsection \ref{Subsection:DukeEstimate}).
Thus, we add a rapidly decaying term to Schoenfeld's estimate
\cite[Cor.~2]{Schoenfeld} so that the final result is valid in a larger range while keeping
the main term essentially intact for large $x$.

\begin{lemma}\label{Lemma:Mertens}
Assuming the Riemann Hypothesis,
\begin{equation}\label{eq:SuperRosser}\qquad
    \bigg|\sum_{p\leq x}\frac{1}{p} - \log\log{x} - M\bigg| < m(x) \quad\text{for}\quad x\geq 1.048,
\end{equation}
in which $M = 0.261497212847642783755\dots$\label{p:M} is the Meissel--Mertens constant and
\begin{equation}\label{eq:mFunction}
    m(x) :=
    \dfrac{1}{\sqrt{x}} \left(\dfrac{3\log{x} + 4}{8\pi}\right)+ \frac{5}{x^2}.
\end{equation}
\end{lemma}

\begin{proof}
For $x \geq 13.5$, the result follows from \cite[Cor.~2, eq.~(6.21)]{Schoenfeld}
(the summand $5/x^2$ is not needed in this range).
For $1.048 \leq x < 13.5$, the desired inequality 
can be verified by direct computation; see Figure \ref{Figure:Schoenfeld}.
\end{proof}

\begin{figure}
\centering
\includegraphics[width=0.6\textwidth]{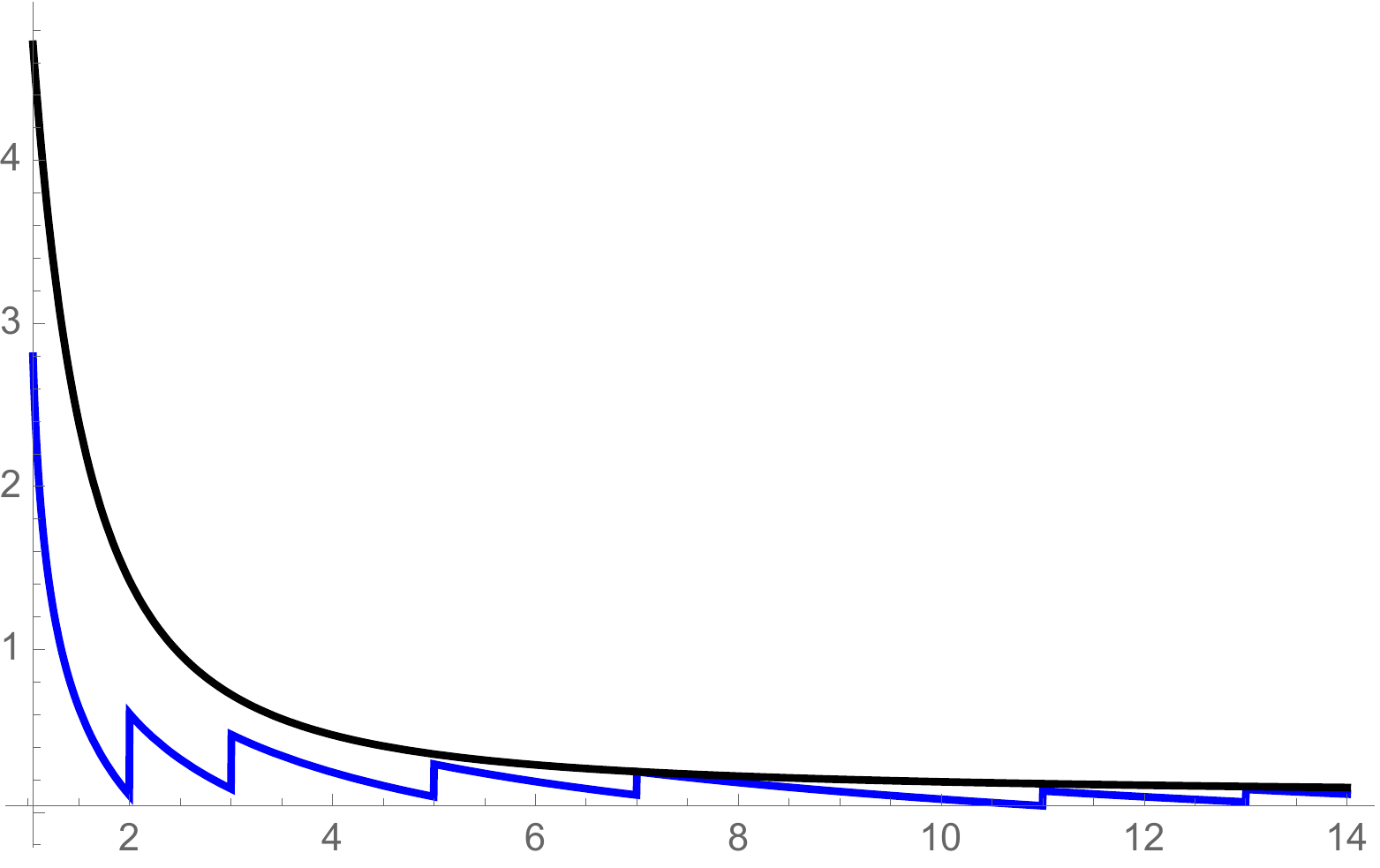}
\caption{Lemma \ref{Lemma:Mertens} in the range not covered by 
\cite[Cor.~2]{Schoenfeld}.  Graph of
$|\sum_{p\leq x} \frac{1}{p} - \log\log x - M|$ (blue) versus
$m(x)$ (black) for $1.048 \leq x \leq 13.5$. The vertical clearances at $x=5,7,11,13$ are approximately
$0.061$, $0.0010$, $0.045$, and $0.018$, respectively.}
\label{Figure:Schoenfeld}
\end{figure}

\subsection{Explicit bounds for Artin $L$-functions}\label{Subsection:DukeEstimate}
The relevant background for what follows can be found in \cite[Ch.~5]{iwaniec2004analytic},
\cite{MR1846449}, or \cite[Chap.~VII]{Neukirch}.
Let $L(s,\chi)$ be an entire Artin $L$-function, in which $\chi$ has degree $d$ and conductor $N$;
it is of  finite order with a canonical product of genus $1$ \cite[Thm.~5.6]{iwaniec2004analytic}.
We assume that $L(s,\chi)$ satisfies GRH, so that  $\log L(s,\chi)$ is analytic on $\Re s > 1/2$.
Moreover, $N \geq 3$ since it is integral and Pizarro-Madariaga
proved that $N \geq (2.91)^d$ \cite[Thm.~3.2]{MR2728993};
Odlyzko's lower bound $(2.38)^d$ \cite[p.~482]{MR0453701} also implies
the desired estimate.

\begin{lemma}\label{Lemma:ZetaDelta}
$| L(s,\chi) | \leq \zeta(\delta)^d$ for $\Re s \geq \delta > 1$.
\end{lemma}

\begin{proof}
For $\Re s > 1$, \cite[(22)]{Duke} ensures that
\begin{equation}\label{eq:ArtinLog}
\log L(s,\chi) = \sum_p \sum_{m=1}^{\infty} \frac{1}{m} \chi(p^m) p^{-ms},
\end{equation}
in which $| \chi(p^m) | \leq d$ \cite[(23)]{Duke}.
Therefore, \eqref{eq:ArtinLog} converges absolutely and
\begin{equation}\label{eq:assembly2}
|\log L(s,\chi)| 
\leq \sum_p \sum_{m=1}^{\infty} \frac{1}{m}| \chi(p^m)| p^{-m \Re s}
\leq d \sum_p\sum_{m=1}^{\infty} \frac{p^{-m \delta}}{m}  
= d \log \zeta(\delta).
\end{equation}
So for $\Re s \geq \delta > 1$, 
$| L(s,\chi) | = |e^{\log L(s,\chi)}| \leq e^{|\log L(s,\chi)|} \leq e^{ d \log \zeta(\delta)} = \zeta(\delta)^d$.
\end{proof}

We need an upper bound for $L(s,\chi)$ on a slightly larger half plane that extends into the
critical strip.
As is customary in the field, we often write $s = \sigma + it$ for a complex variable, in which $\sigma,t \in \R$.
For an analytic function $f(s)$, let $\overline{f}(s) = \overline{f(\overline{s})}$, which is also
analytic. Then 
$\Lambda(s) = \overline{ \Lambda}(1-s)$,
in which $\Lambda(s) = \gamma(s) L(s,\chi)$ and 
\begin{equation*}
\gamma(s) = \epsilon N^{\frac{1}{2}(s- \frac{1}{2})} \prod_{j=1}^d \Gamma_{\R}(s + \mu_j)
\end{equation*}
with $\Gamma_{\R}(s) = \pi^{-s/2} \Gamma(s/2)$ and for some $|\epsilon|=1$
and $\mu_i \in \{0,1\}$ \cite[Sect.~1.3]{Booker}.  In fact, $\mu_i = 0$ exactly $\frac{1}{2}(d+\ell)$ times and $\mu_i=1$
exactly $\frac{1}{2}(d-\ell)$ times, where $\ell$ is the value of the underlying character on
complex conjugation \cite[p.~111]{Duke}.

\begin{lemma}\label{Lemma:BoundL}
$\displaystyle|L(s,\chi)| \leq C_1^d \sqrt{N} |1+s|^{\frac{d}{2}}$ for $\Re s \geq \frac{1}{2}$,
where $C_1 = \tfrac{\zeta( 3/2)}{\sqrt{2\pi}} \approx 1.04219$.\label{p:C1}
\end{lemma}

\begin{proof}
The analytic conductor
\begin{equation*}
Q(s) = N \prod_{j=1}^d \frac{s + \mu_j}{2\pi}
\end{equation*}
satisfies
\begin{equation}\label{eq:QBound}
|Q(s)| \leq \frac{N}{(2\pi)^d} |1+s|^d
\quad \text{for $\sigma \geq - \tfrac{1}{2}$}.
\end{equation}
Let $X(s) =  \overline{ \gamma}(1-s) / \gamma(s)$,
so that the functional equation becomes
$L(s,\chi) = X(s) \overline{L}(1-s,\chi)$ \cite[p.~387]{Booker} (note that \cite{Booker} uses $\chi$ for our $X$;
we have already reserved $\chi$ for the character of $L$).
A special case of \cite[Lem.~4.1]{Booker} 
implies that
\begin{equation*}
|L(s,\chi)|^2 \leq  | X(s) Q(s)| \sup_{\sigma = \frac{3}{2}} | L(s,\chi)|^2
\qquad \text{for $- \tfrac{1}{2} \leq \sigma \leq \tfrac{3}{2}$}.
\end{equation*}
Since $|X(s)|=1$ for $\sigma = \frac{1}{2}$, the previous inequality, \eqref{eq:QBound},
and Lemma \ref{Lemma:ZetaDelta} yield
\begin{equation}\label{eq:Phrag1}\qquad
|L(s,\chi)| \leq \zeta( \tfrac{3}{2})^d \sqrt{ |Q(s)|} 
\qquad\text{for $\sigma= \tfrac{1}{2}$}.
\end{equation}
Lemma \ref{Lemma:ZetaDelta} and the definition of $Q$ ensure that
\begin{equation}\label{eq:Phrag2}
|L(s,\chi)| \leq \zeta(2\pi)^d \leq \zeta( \tfrac{3}{2})^d  \sqrt{ |Q(s)|}
\qquad\text{for $\sigma\geq 2\pi$}.
\end{equation}
Since $L / \sqrt{Q}$ is analytic for $\sigma \geq \frac{1}{2}$,
the Phragm\'en--Lindel\"of principle and \eqref{eq:QBound} imply
that the desired inequality holds for $\sigma \geq \frac{1}{2}$.
\end{proof}

\begin{remark}
Let us be more explicit about the final step in the proof of Lemma \ref{Lemma:BoundL}.
As mentioned above, $L(s,\chi)$ is an entire function of finite order \cite[Thm.~5.6]{iwaniec2004analytic}.  Thus,
$|L(s,\chi)| = O(\exp|t|^{\alpha})$ for some $\alpha \geq 1$ and all $s\in \C$; this growth estimate permits us to
appeal to the Phragm\'en--Lindel\"of principle below.

Suppose that $\frac{1}{2} \leq \Re s \leq \sigma_0$, in which $\sigma_0 \geq 2 \pi$.
The bounds \eqref{eq:Phrag1} and \eqref{eq:Phrag2} ensure that $L/\sqrt{Q}$ is bounded by $\zeta(3/2)^d$
on the boundary of the vertical strip $\{ \sigma + it : \frac{1}{2} \leq \sigma \leq \sigma_0\}$.
Consequently, the Phragm\'en--Lindel\"of principle for a strip \cite[Thm.~8.1.3]{MurtyAnalytic}
ensures that the same bound holds on the interior. In particular, it holds at $s$,
which was an arbitrary point with $\Re s \geq \frac{1}{2}$.
\end{remark}

\subsection{An explicit bound for the logarithmic derivative}\label{Subsection:DukeLogDerivative}
In what follows, let
\begin{equation*}\label{p:eta}
1 < \eta \leq \tfrac{3}{2} \qquad \text{and} \qquad 0 < \delta < \tfrac{1}{2}.
\end{equation*}
These constants will be optimized toward the end of the proof.  In particular,
the range of admissible $\delta$ shall be further restricted as new information emerges.

\begin{lemma}\label{Lemma:LL384}
Let $L(s,\chi)$ be an entire Artin $L$-function that satisfies GRH, in which 
$\chi$ has degree $d$ and conductor $N$.  
For $\tfrac{1}{2}+\delta \leq \sigma \leq 2\eta - \tfrac{1}{2} -\delta$,
\begin{equation}\label{eq:TTLL}
\bigg| \frac{L'(s,\chi)}{L(s,\chi)} \bigg| 
\leq  C_2 d \log\big( C_3 N^{\frac{1}{d}} (|t|+4)  \big),
\end{equation}
where
\begin{equation}\label{eq:C2-3}
C_2:= C_2(\delta,\eta) = \frac{2\eta-1}{2\delta^2}
\qquad \text{and} \qquad
C_3:= C_3(\eta) =  \left(C_1 \frac{\zeta(\eta)}{\zeta(2\eta)} \right)^2.
\end{equation}
\end{lemma}

\begin{proof}
Since $\zeta(\sigma)/\sqrt{\sigma+1}$ is decreasing for $\sigma>1$,
computation confirms that
$\zeta(\sigma) \leq C_1 \sqrt{\sigma+1}$ for $\sigma \geq 1.89$.
For $\sigma \geq 2$, Lemma \ref{Lemma:ZetaDelta} and the integrality of $N$ imply that
\begin{equation*}
|L(s,\chi)| \leq \zeta(2)^d \leq C_1^d  (2+1)^{\frac{d}{2}}
\leq C_1^d \sqrt{N} ( | t| + 3)^{\frac{d}{2}}.
\end{equation*}
On the other hand, Lemma \ref{Lemma:BoundL} ensures that 
\begin{equation*}
|L(s,\chi)| \leq C_1^d \sqrt{N}|1+s|^{\frac{d}{2}}
\leq C_1^d \sqrt{N}(|t|+3)^{\frac{d}{2}}
\end{equation*}
for $\frac{1}{2} \leq \sigma \leq 2$.  Therefore,
\begin{equation*}
\Re \log L(s,\chi) 
= \log|L(s,\chi)|
< d \log C_1 + \tfrac{1}{2}\log N + \tfrac{d}{2} \log \left( | \Im s|+3\right)
\end{equation*}
for $\sigma > \frac{1}{2}$.  Next observe that \cite[Lem.~4.5]{Booker} with $\theta =0$ provides
\begin{equation*}
\left( \frac{\zeta(2\eta)}{\zeta(\eta)} \right)^d \leq \,  |L(\eta+it,\chi)| 
\end{equation*}
since $\eta > 1$.  Therefore,
\begin{equation*}
\Re \log L(\eta+it,\chi) = \log |L(\eta+it,\chi)|
\geq d\log \frac{\zeta(2\eta)}{\zeta(\eta)}.
\end{equation*}
For $f$ analytic on $|z| < R$,  \cite[Cor.~5.3]{Kresin} says that
\begin{equation*}
|f'(z)| \leq \frac{2R}{(R-|z|)^2} \sup_{|\xi|< R} \Re (f(\xi)-f(0)).
\end{equation*}
Fix $t\in \R$ and apply the inequality above to
$f(z) = \log L(z+ \eta+it,\chi)$
with
\begin{equation*}
R = \eta-\tfrac{1}{2}   \qquad \text{and} \qquad  |\underbrace{s  - (\eta+it) }_z| \leq R - \delta
\end{equation*}
so that 
\begin{equation*}
\tfrac{1}{2}+\delta 
\,\,=\,\, \eta - R + \delta 
\,\,\leq\,\, \Re s
\,\,\leq\,\, \eta + R - \delta 
\,\,=\,\,
 2 \eta - \tfrac{1}{2} - \delta.
\end{equation*}
As $z$ ranges over $|z| \leq R - \delta$, observe that
$s$ assumes every value in the horizontal segment
$[\frac{1}{2}+\delta + it ,\,  2 \eta - \tfrac{1}{2} - \delta + it]$;
see Figure \ref{Figure:Circles}.  
Since $R = \eta - \frac{1}{2} < \frac{3}{2} - \frac{1}{2} = 1$, it follows that
$|\Im \xi| < 1$ whenever $|\xi| < R$.  Therefore,
\begin{align*}
    \bigg| \frac{L'(s,\chi)}{L(s,\chi)} \bigg|
    &\leq \frac{2(\eta-\frac{1}{2})}{  (R - (R-\delta))^2} 
    \sup_{|\xi| < R}  \Re \big(\log L(\xi + \eta + it,\chi) - \log L(\eta+it,\chi)\big) \\
    &\leq \frac{2\eta-1}{\delta^2}
    \sup_{ |\xi| < R} \big(d \log C_1 + \tfrac{1}{2}\log N + \tfrac{d}{2} \log \left( |t+\Im \xi|+3\right)- 
    \log L(\eta+it,\chi) \big)\\
    &\leq \frac{2\eta-1}{\delta^2}   \left(d \log C_1 + \tfrac{1}{2}\log N + \tfrac{d}{2} \log \left( |t|+4\right)- 
    d \log \frac{\zeta(2\eta)}{\zeta(\eta)}\right)\\
    &\leq \frac{(2\eta-1)d}{2\delta^2}      \left(2 \log C_1 + \log N^{\frac{1}{d}} +  \log \left( |t|+4\right)- 
    2\log\frac{\zeta(2\eta)}{\zeta(\eta)}  \right)\\
    &\leq \frac{(2\eta-1)d}{2\delta^2} \log\left[ C_1^2 N^{\frac{1}{d}} \left(|t|+4\right) 
    \left( \frac{\zeta(\eta)}{\zeta(2\eta)} \right)^2 \right]\\
    &\leq C_2 d \log\big( C_3 N^{\frac{1}{d}} (|t|+4)  \big).
\end{align*}
Since $t \in \R$ is arbitrary, the desired bound holds
for $\frac{1}{2}+\delta \leq \Re s \leq 2\eta - \frac{1}{2} -\delta$.
\end{proof}

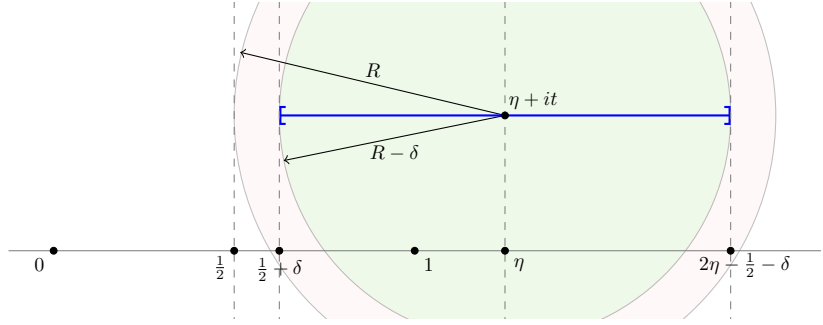
\begin{figure}
    \begin{tikzpicture}[scale=0.6, every node/.style={scale=0.75}]
        \clip(-5,-1.5) rectangle (13.5,5.5);

        \filldraw[fill=red!10!white, opacity=0.25] (6,3) circle (6cm);
        \filldraw[fill=green!25!white,opacity=0.25] (6,3) circle (5cm);
        \draw[->] (6,3)--(0.135,4.4) node[midway, above]{$R$};
        \draw[->] (6,3)--(1.1,2) node[midway, below]{$R-\delta$};
        \draw[blue,thick,{[-]}] (1,3)--(11,3);
       
       \draw[thin,gray](-5,0)--(13,0);
       \draw[thin,gray,dashed](0,-2)--(0,6);
       \draw[thin,gray,dashed](1,-2)--(1,6);
       \draw[thin,gray,dashed](6,-2)--(6,6);
       \draw[thin,gray,dashed](11,-2)--(11,6);

        \filldraw[fill=black] (-4,0) circle (0.075cm)node[xshift=-0.25cm,yshift=-0.25cm]{$0$};
        \filldraw[fill=black] (0,0) circle (0.075cm)node[xshift=-0.25cm,yshift=-0.3cm]{$\frac{1}{2}$};
        \filldraw[fill=black] (1,0) circle (0.075cm)node[below]{$\frac{1}{2}+ \delta$};
        \filldraw[fill=black] (4,0) circle (0.075cm)node[xshift=.25cm,yshift=-0.25cm]{$1$};
        \filldraw[fill=black] (6,0) circle (0.075cm)node[xshift=0.25cm,yshift=-0.25cm]{$\eta$};
        \filldraw[fill=black] (11,0) circle (0.075cm)node[xshift=0.25cm,yshift=-0.25cm]{$2\eta-\frac{1}{2}-\delta$};

        \filldraw[fill=black] (6,3) circle (0.075cm)node[above, xshift=0.5cm]{$\eta+it$};
    \end{tikzpicture}
\caption{For $t \in \R$, the inequality \eqref{eq:TTLL}
holds for $s$ in the green region.  In particular, it holds for all $s$ on the blue line segment.  Since $t \in \R$ is arbitrary, \eqref{eq:TTLL} holds for
$\Re s \in [ \frac{1}{2} + \delta, 2 \eta - \frac{1}{2} - \delta]$.}
\label{Figure:Circles}
\end{figure}

Duke's approach to the previous lemma uses an estimate result of Littlewood \cite[Lem.~4]{Littlewood}, which is \cite[Lem.~2]{Duke} in Duke's paper, instead of the slightly sharper \cite[Cor.~5.3]{Kresin} used above. More information about ``sharp real part theorems'' for the derivative can be found in \cite[Ch.~5]{Kresin}, along with a host of historical references.

\subsection{An integral estimate}\label{Subsection:Gamma}
The proof of Theorem \ref{Theorem:Duke} requires an integral estimate that involves the gamma function.
First, consider the real-valued function
\begin{equation}\label{eq:fFunction}
f(s) = \frac{\sqrt{2\pi}|s+1|^{\sigma+\frac{1}{2}}}{|s|}\exp\bigg(\frac{1}{6|s+1|}\bigg)
\end{equation}
for $s = \sigma+ it$ with $\sigma \in [-1+\delta , -\frac{1}{2} + \delta]$.  Since
\begin{equation*}
0 \leq f(s)  \leq  \frac{ \sqrt{2\pi} e^{\frac{1}{6 \delta}}(\tfrac{1}{2}+\delta+|t|)^{\delta}}{|t|} \to 0
\end{equation*}
uniformly as $|t| \to \infty$ for such $\sigma$, we may use numerical methods to maximize
$f(s)$ in the vertical strip $\sigma \in [-1+\delta , -\frac{1}{2} + \delta]$.  
Let 
\begin{equation}\label{eq:tau}
\tau = 0.219733068786773\ldots
\end{equation}
denote the unique root of
$f(-1+\delta) - f(-\tfrac{1}{2}+\delta)$ with $0 < \delta < \tfrac{1}{2}$;
see Figure \ref{Figure:Tau}.  Computation says that 
for $0 < \delta \leq \tau$, the function $f(s)$ attains its maximum value
\begin{equation}\label{eq:mu}
\mu =
\dfrac{\sqrt{2 \pi }  \delta ^{\delta -\frac{1}{2}} e^{\frac{1}{6 \delta}}} {1- \delta  }
\end{equation} 
at $s=-1+\delta$; see Figure \ref{Figure:Surface}.

\begin{figure}
\centering
\begin{subfigure}[t]{0.475\textwidth}
\centering
\includegraphics[width=1\textwidth]{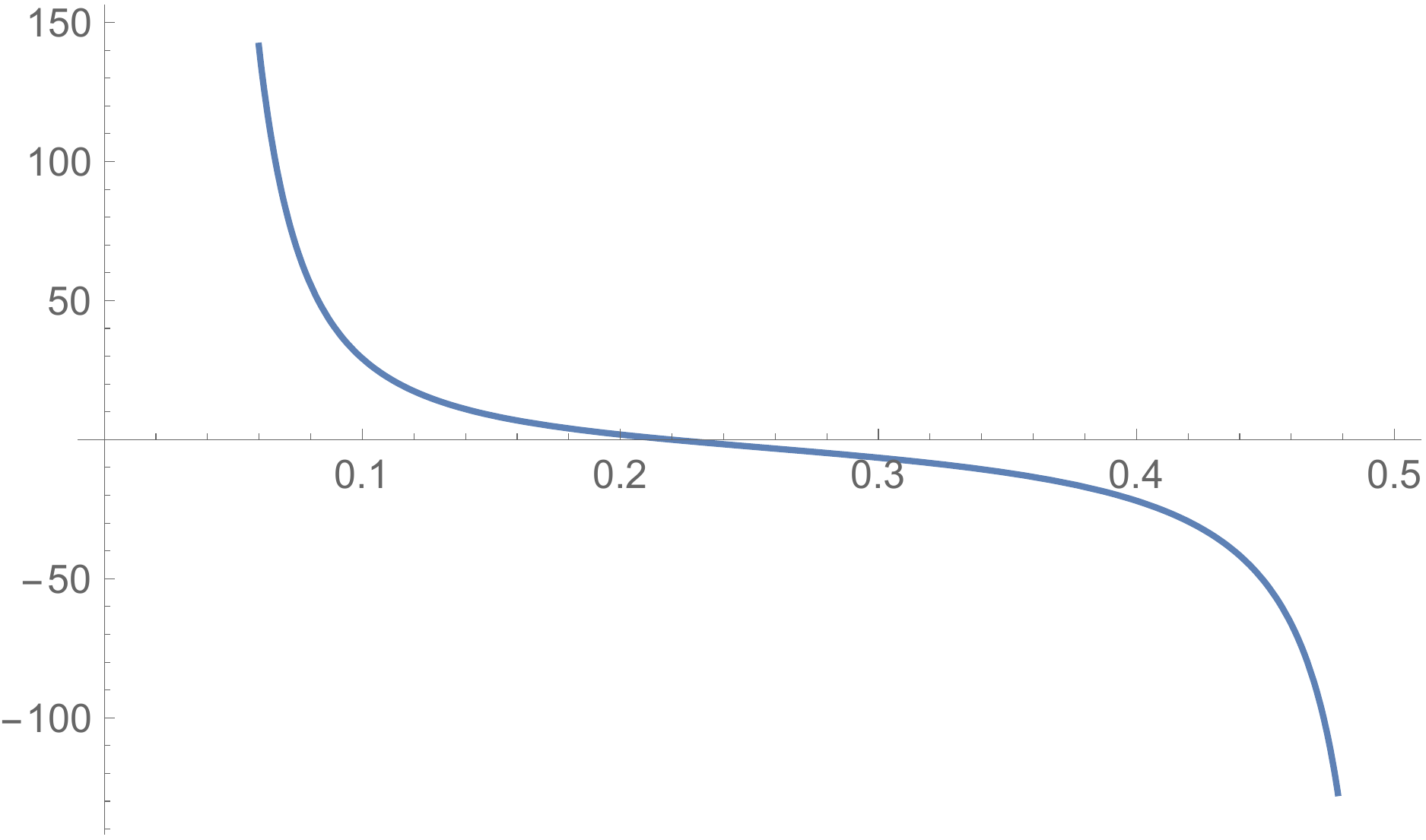}
\caption{$0 < \delta < \frac{1}{2}$}
\label{Figure:Tau1}
\end{subfigure}
\quad
\begin{subfigure}[t]{0.475\textwidth}
\centering
\includegraphics[width=1\textwidth]{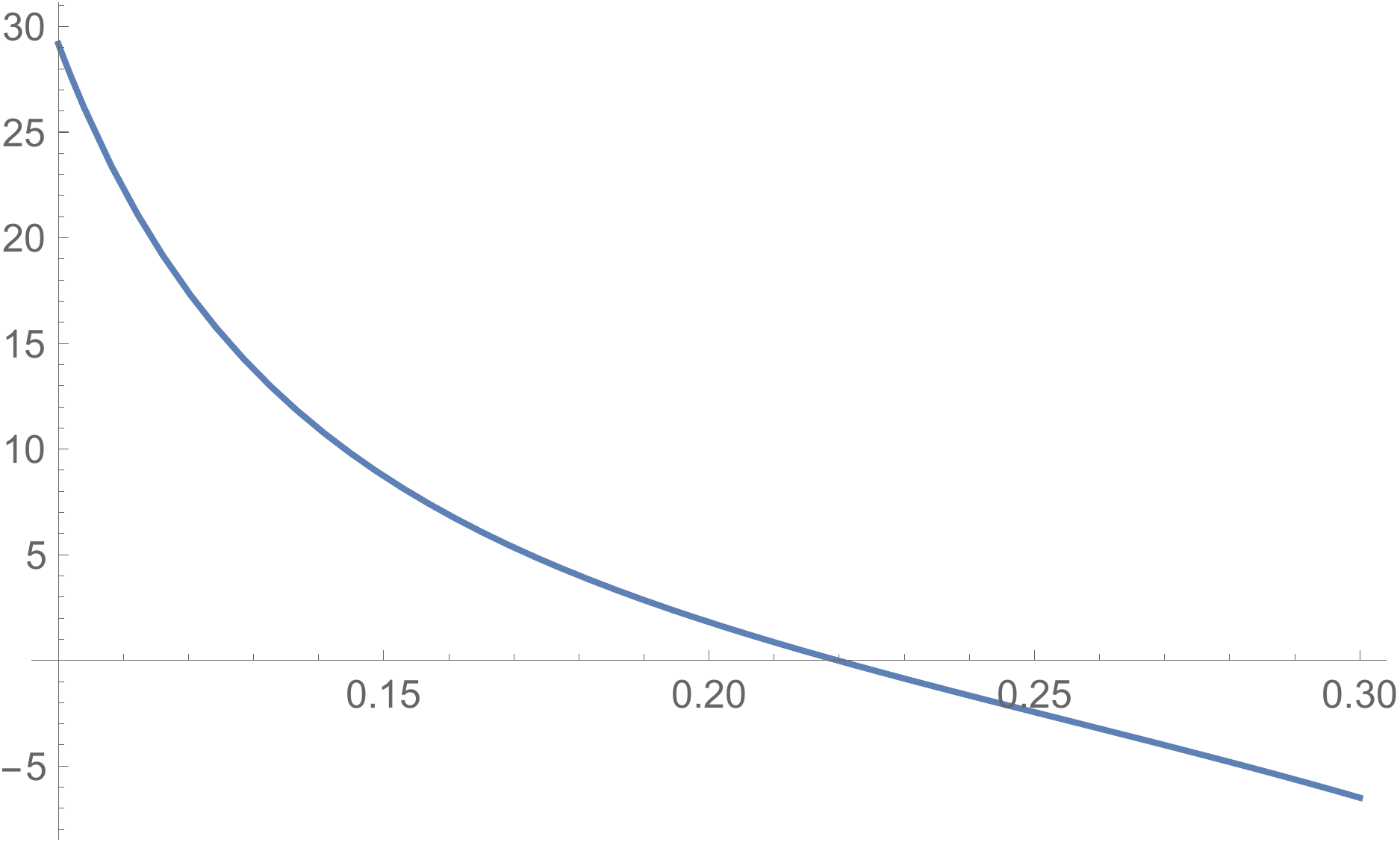}
\caption{$0.1 < \delta < 0.3$}
\label{Figure:Tau2}
\end{subfigure}
\caption{For $0 < \delta \leq \frac{1}{2}$, the function $f(-1+\delta) - f(-\tfrac{1}{2}+\delta)$ 
has a unique real root $\tau = 0.219733068786773\ldots$.}
\label{Figure:Tau}
\end{figure}

\begin{lemma}\label{Lemma:GammaLog}
For $0 < \delta \leq \tau$ and $-1+\delta \leq \sigma \leq -\frac{1}{2} + \delta$,
\begin{equation*}
 \int_{-\infty}^{\infty}|\Gamma(\sigma+it)| 
 \log\big( C_3 N^{\frac{1}{d}} (|t|+4 ) \big) \,dt 
 \,\leq\, \frac{C_4 \log N}{d} + C_5,
\end{equation*}
in which
\begin{align}
C_4 &:= C_4(\delta) =
\frac{ 4 \sqrt{2} \, \delta ^{\delta -\frac{1}{2}} e^{\frac{1}{6 \delta }} }{\sqrt{\pi}(1-\delta )}
 \qquad \text{and} \label{eq:C4} \\
C_5 &:= C_5(\delta, \eta) = C_4\Big(\log  C_3 + \frac{\pi}{2} \Big).\label{eq:C5}
\end{align}
\end{lemma}

\begin{proof}
For $s = \sigma+it$ with $\sigma \geq 0$, we have the inequality \cite[5.6E9]{NIST}:
\begin{equation*}
|\Gamma(s)| \leq \sqrt{2\pi} |s|^{\sigma - \frac{1}{2}} e^{-\pi|t|/2} \exp\bigg(\frac{1}{6|s|}\bigg).
\end{equation*}
Since $\Gamma(s+1) = s \Gamma(s)$, the previous inequality 
and \eqref{eq:fFunction} yield
\begin{equation}\label{eq:catspyjamaslabel1}
    |\Gamma(s)| = \frac{|\Gamma(s+1)|}{|s|}  
    \leq \frac{ \sqrt{2\pi} |s+1|^{\sigma + \frac{1}{2}} e^{-\pi|t|/2} }{|s|}\exp\bigg(\frac{1}{6|s+1|}\bigg)
    =f(s)e^{-\pi|t|/2} .
\end{equation}
For $0 < \delta \leq \tau$,
\eqref{eq:mu} and \eqref{eq:catspyjamaslabel1} imply that
\begin{equation*}\qquad
    |\Gamma(\sigma+it)|
    \leq \mu e^{-\pi|t|/2} \quad  \text{for $\sigma \in [-1+\delta ,  - \tfrac{1}{2}+\delta]$ and $t \in \R$}.
\end{equation*}
Since $|\Gamma(z)| = |\Gamma(\overline{z})|$ for $z \in \C$, we get
\begin{align*}
&\int_{-\infty}^{\infty}|\Gamma(\sigma+it)|  \log\big( C_3 N^{\frac{1}{d}} (|t|+4 ) \big) \,dt \\
&\qquad= 2\int_{0}^{\infty}|\Gamma(\sigma+it)|  \log\big( C_3 N^{\frac{1}{d}} \big(t+4 \big) \big) \,dt \\
&\qquad\leq  2 \mu \int_0^{\infty} e^{-\pi t/2} \log\big( C_3 N^{\frac{1}{d}} \big(t+4 \big) \big)  \,dt \\
&\qquad\leq  2 \mu \int_0^{\infty} e^{-\pi t/2} \left[\log C_3 + \frac{1}{d} \log N +  \log\left(t+4\right) \right]  \,dt \\
&\qquad\leq  2 \mu\left( \frac{2}{\pi}\log C_3  + \frac{2}{\pi}\cdot\frac{\log N}{d}
+1\right) \\
&\qquad\leq  \frac{4 \mu}{\pi} \left( \frac{\log N}{d} + \log C_3+\frac{\pi}{2} \right) \\
&\qquad = C_4 \frac{\log N}{d} +C_5.\qedhere
\end{align*}
\end{proof}

\begin{figure}
\centering
\begin{subfigure}[t]{0.475\textwidth}
\centering
\includegraphics[width=1\textwidth]{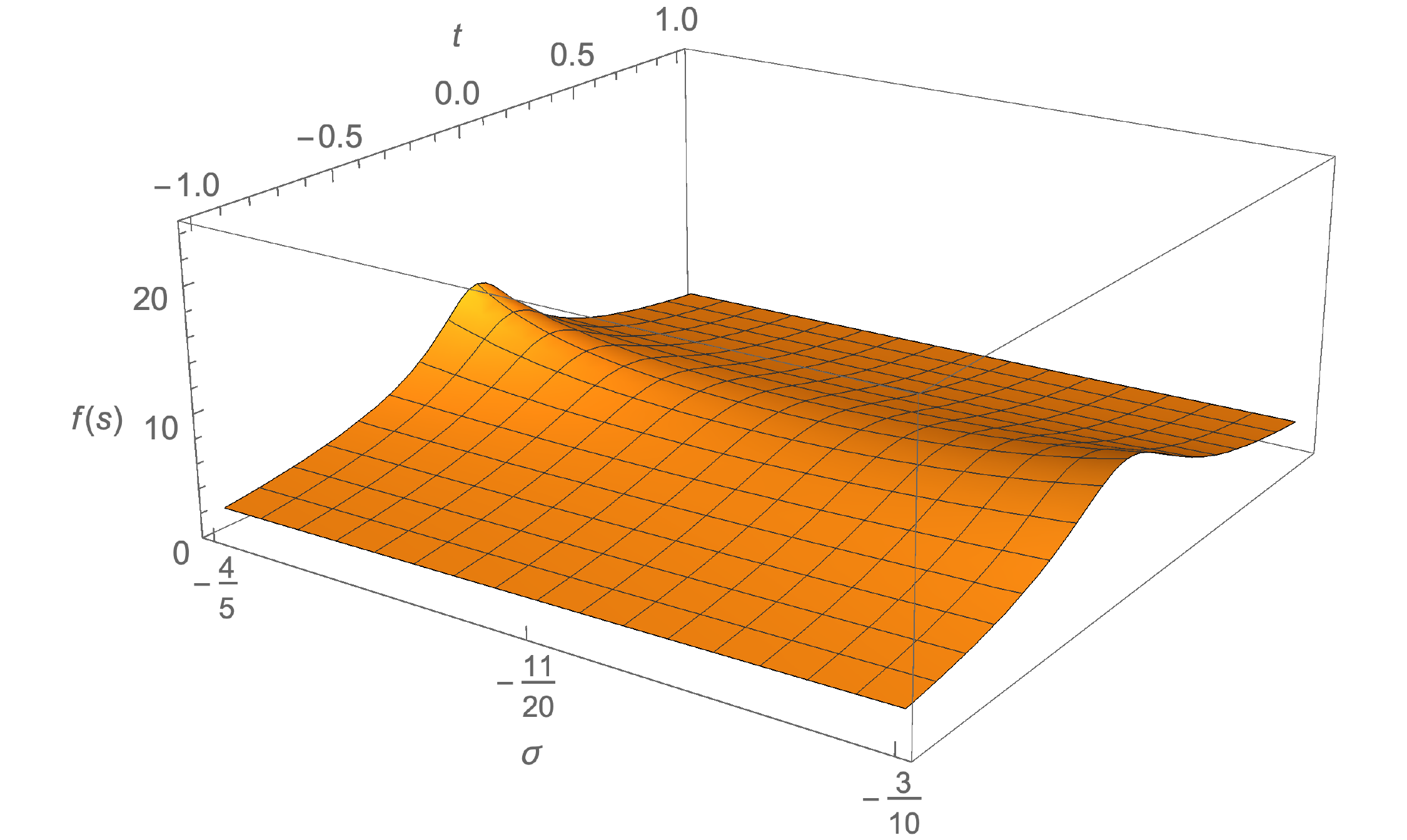}
\caption{$\delta = \frac{1}{5}$}
\label{Figure:F1}
\end{subfigure}
\quad
\begin{subfigure}[t]{0.475\textwidth}
\centering
\includegraphics[width=1\textwidth]{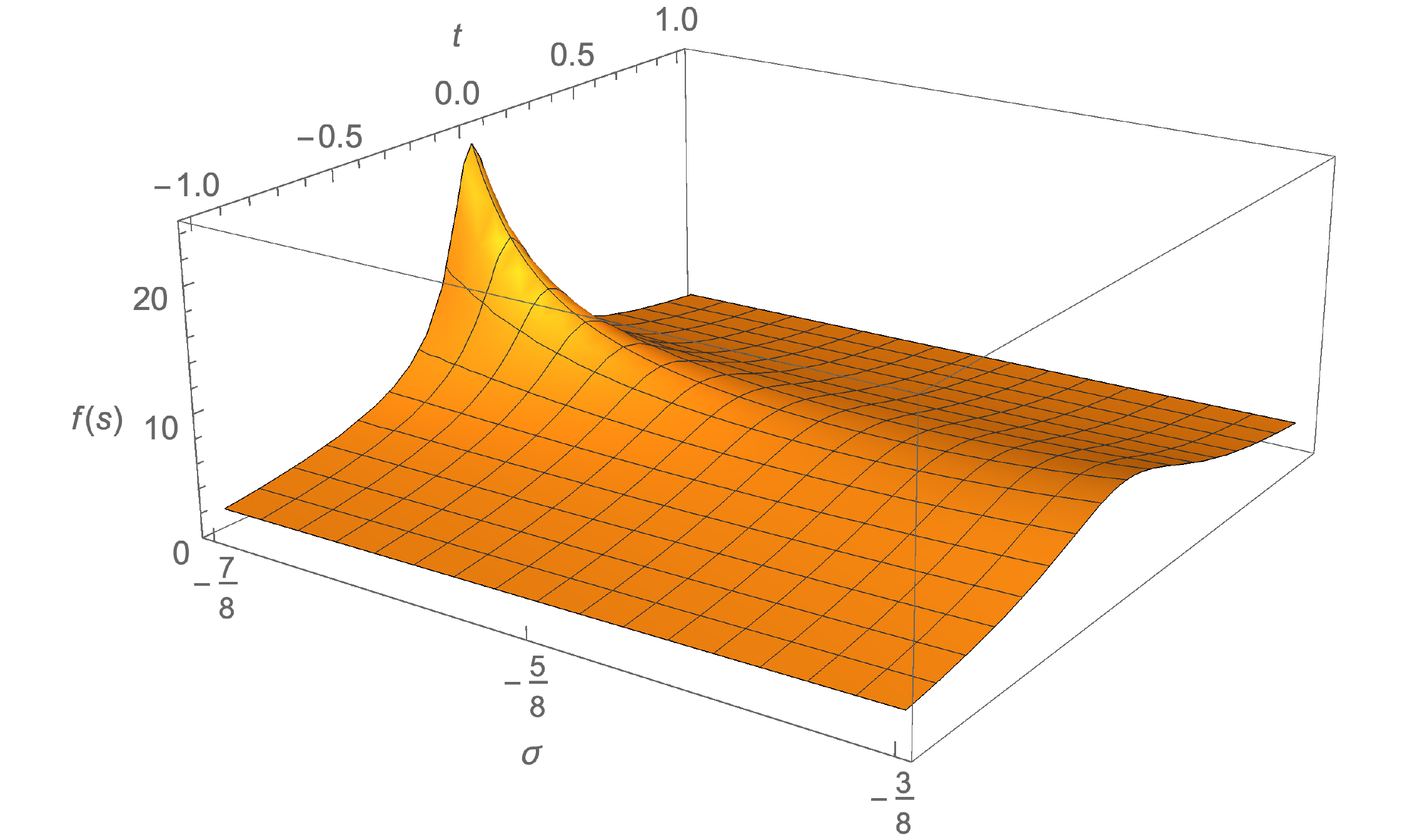}
\caption{$\delta = \frac{1}{8}$}
\label{Figure:F2}
\end{subfigure}
\caption{For $0 < \delta \leq \tau =0.219733068786773\ldots$, 
the nonnegative, real-valued function $f(s)$ attains its maximum value in the vertical strip
$\sigma \in [-1 + \delta , -\frac{1}{2}+\delta]$ at $s = -1+\delta$.}
\label{Figure:Surface}
\end{figure}

\subsection{Approximating the logarithmic derivative}\label{Subsection:Approx}
The next lemma is an explicit version of \cite[Lem.~4]{Duke}, which relates the logarithmic derivative of an
entire Artin $L$-function to a sum over the primes.

\begin{lemma}\label{Lemma:LogChiEstimate}
Let $L(s,\chi)$ be an entire Artin $L$-function that satisfies GRH, in which 
$\chi$ has degree $d$ and conductor $N$.  
For $1 \leq u \leq \frac{3}{2}$ and $x>1$,\label{p:u}
\begin{equation*}
    \bigg| \sum_p (\log p) \chi(p) p^{-u} e^{-p/x} +  \frac{L'(u,\chi)}{L(u,\chi)} \bigg|
    \,\leq\,  \frac{C_2 d x^{\frac{1}{2}+\delta-u}}{2\pi} \bigg(  \frac{C_4 \log N}{d} + C_5 \bigg)  + 0.76d .
\end{equation*}
\end{lemma}

\begin{proof}
For $\Re s > 1$, the derivative of \eqref{eq:ArtinLog} provides
\begin{equation}\label{eq:LogLLSum}
\sum_p \log p \sum_{m=1}^{\infty} \chi(p^m) p^{-ms}
=-\frac{L'(s,\chi)}{L(s,\chi)} .
\end{equation}
Substitute $y = p^m/x$ in the Cahen--Mellin integral (see \cite{HardyIntegral} or \cite[6.6.2, p.~380]{MurtyAnalytic})
\begin{equation*}\qquad
    e^{-y} = \frac{1}{2\pi i} \int_{\eta-i\infty}^{\eta+i\infty} y^{-s} \Gamma(s)\,ds
\qquad\text{for $y>0$},
\end{equation*}
and obtain 
\begin{equation*}\qquad
    e^{-p^m/x}  = \frac{1}{2\pi i} \int_{\eta-i\infty}^{\eta+i\infty} p^{-ms} x^s\Gamma(s)\,ds
\qquad \text{for $x>0$}.
\end{equation*}
For $1 \leq u \leq \frac{3}{2}$ and $x>1$, it follows from \eqref{eq:LogLLSum} that
\begin{align*}
&\sum_p \log p \sum_{m=1}^{\infty} \chi(p^m) p^{-mu} e^{-p^m/x} \\
&\qquad = \sum_p \log p \sum_{m=1}^{\infty} \chi(p^m) p^{-mu} 
\bigg(\frac{1}{2\pi i }\int_{\eta-i\infty}^{\eta+i\infty} p^{-ms} x^s\Gamma(s)\,ds \bigg)\\
&\qquad = \frac{1}{2\pi i }\int_{\eta-i\infty}^{\eta+i\infty} \bigg( \sum_p \log p \sum_{m=1}^{\infty} \chi(p^m) p^{-m(s+u)} \bigg)
x^s\Gamma(s)\,ds \\
&\qquad =- \frac{1}{2\pi i }\int_{\eta-i\infty}^{\eta+i\infty} \frac{L'(s+u,\chi)}{L(s+u,\chi)} x^s\Gamma(s)\,ds.
\end{align*}
Shift the integration contour to the vertical line $\Re s = \frac{1}{2}+\delta-u$.  Since $\Gamma$ has a simple pole with residue $1$ at $s=0$ and $\frac{1}{2}+\delta - u < 0$, we pick up the residue
\begin{equation*}
\underset{s=0}{\operatorname{Res}}\bigg(\frac{L'(s+u,\chi)}{L(s+u,\chi)} x^s\Gamma(s)\bigg) = \frac{L'(u,\chi)}{L(u,\chi)}  
\end{equation*}
and obtain
\begin{equation}\label{eq:DukeFundamental}\small
\sum_p \log p \sum_{m=1}^{\infty} \frac{ \chi(p^m) }{ p^{mu} e^{p^m/x} }
 +  \frac{L'(u,\chi)}{L(u,\chi)}  
= - \frac{1}{2\pi i }\int_{\frac{1}{2}+\delta-u-i\infty}^{\frac{1}{2}+\delta-u+i\infty} \frac{L'(s+u,\chi)}{L(s+u,\chi)} x^s\Gamma(s)\,ds.
\end{equation}
Since $\frac{1}{2}+\delta-u \in [-1 + \delta, -\frac{1}{2}+\delta]$,
we estimate the integral on the right-hand side of \eqref{eq:DukeFundamental} 
with Lemmas \ref{Lemma:LL384} and \ref{Lemma:GammaLog}:
\begin{align*}
&\left|\int_{\frac{1}{2}+\delta-u-i\infty}^{\frac{1}{2}+\delta-u+i\infty} \frac{L'(s+u,\chi)}{L(s+u,\chi)} x^s\Gamma(s)\,ds \right| \\
&\qquad\qquad\leq \int_{-\infty}^{\infty} \bigg| \frac{L'(\frac{1}{2}+\delta+it,\chi)}{L(\frac{1}{2}+\delta+it,\chi)}\bigg| 
x^{\frac{1}{2}+\delta-u} |\Gamma(\tfrac{1}{2}+\delta-u+it)| dt \\
&\qquad\qquad\leq C_2 d  x^{\frac{1}{2}+\delta-u}  \int_{-\infty}^{\infty}|\Gamma(\tfrac{1}{2}+\delta-u+it)| \log\big(C_3 N^{\frac{1}{d}}(|t|+4)\big) dt \\
&\qquad\qquad\leq C_2 d x^{\frac{1}{2}+\delta-u} \bigg(  \frac{C_4 \log N}{d} + C_5 \bigg).
\end{align*}
The triangle inequality, \eqref{eq:DukeFundamental}, and the preceding inequality imply that
\begin{align*}
    &\left| \sum_p \log p \frac{ \chi(p) }{ p^{u} e^{p/x} } +  \frac{L'(u,\chi)}{L(u,\chi)} \right|
    \nonumber\\
    &\hspace{1.5cm}\leq \left| \sum_p \log p \sum_{m=1}^{\infty} \frac{ \chi(p^m) }{ p^{mu} e^{p^m/x} } +  \frac{L'(u,\chi)}{L(u,\chi)} \right|
    + \left| \sum_p \log p \sum_{m=2}^{\infty} \frac{ \chi(p^m) }{ p^{mu} e^{p^m/x} }\right|\nonumber\\
    &\hspace{1.5cm}\leq \frac{ C_2 d x^{\frac{1}{2}+\delta-u} }{2\pi} \bigg(  \frac{C_4 \log N}{d} + C_5 \bigg) + \left| \sum_p \log p \sum_{m=2}^{\infty} \frac{ \chi(p^m) }{ p^{mu} e^{p^m/x} }\right|.
\end{align*}

We bound the sum over $m \geq 2$ above using $| \chi(p^m) | \leq d$, the inequality $1 \leq u \leq \frac{3}{2}$, and numerical summation:
\begin{align*}
\bigg|\sum_p \log p \sum_{m=2}^{\infty} \frac{ \chi(p^m) }{ p^{mu} e^{p^m/x} } \bigg|
&\,\leq\, d \sum_p \log p \sum_{m=2}^{\infty} p^{-mu}  
\,\leq\, d \sum_p \frac{\log p}{p^u(p^u-1)}  \\
& \,\leq\, d \sum_p \frac{\log p}{p(p-1)} = (0.7553666\ldots )d
\,<\, 0.76d.\qedhere
\end{align*}
\end{proof}

\subsection{An exponentially weighted sum}\label{Subsection:Weighted}

The next step in the proof of Theorem \ref{Theorem:Duke} is 
an explicit version of the argument at the top of \cite[p.~113]{Duke}.  
In what follows, $f(t) = O^{\star}(g(t))$ means that $|f(t)| \leq |g(t)|$ for all 
$t$ under consideration.  That is, $O^{\star}$ is like Landau's Big-$O$ notation, except that the implicit constant is always $1$.

\begin{lemma}\label{Lemma:L1N}
Let $L(s,\chi)$ be an entire Artin $L$-function that satisfies GRH, in which  $\chi$ has degree $d$ and conductor $N$.  
For $x>1$,
\begin{equation*}
\Bigg| \log L(1,\chi) - \sum_p \chi(p) p^{-1}e^{-p/x} \Bigg|
\, \leq \,  2.19d +  \frac{ C_6 \log N + C_7 d}{x^{\frac{1}{2}-\delta} \log x},
\end{equation*}
in which
\begin{equation}\label{eq:C6-7}
C_6 := C_6(\delta,\eta)= \frac{C_2 C_4}{2\pi}
\qquad\text{and}\qquad
C_7 := C_7(\delta,\eta) 
=  \frac{C_2 C_5}{2 \pi}.
\end{equation}
\end{lemma}

\begin{proof}
Fix $x>1$.
We begin by integrating the expressions in the inequality from Lemma \ref{Lemma:LogChiEstimate} over $u \in [1,\frac{3}{2}]$ and obtain
\begin{align*}
    \int_1^{\frac{3}{2}}\sum_p (\log p) \chi(p) p^{-u} e^{-p/x}\,du
    &= \sum_p (\log p)\chi(p) e^{-p/x} \int_1^{\frac{3}{2}} p^{-u} \,du \\
    &= \sum_p \chi(p) e^{-p/x}   \frac{\sqrt{p}-1}{p^{3/2} }\\
    &= \sum_p \chi(p) p^{-1}e^{-p/x} -\sum_p \chi(p) p^{-3/2}e^{-p/x} \\
    &= \sum_p \chi(p) p^{-1}e^{-p/x} +O^{\star}\bigg(d \sum_p \frac{1}{p^{3/2}}\bigg) \\
    &= \sum_p \chi(p) p^{-1}e^{-p/x} + d P(\tfrac{3}{2})O^{\star}(1),
\end{align*}
in which $P(x)$ denotes the prime zeta function.
The initial interchange of integral and summation is permissible by uniform convergence.
Using \eqref{eq:assembly2}, we obtain
\begin{equation*}
    \int_1^{\frac{3}{2}} \frac{L'(u,\chi)}{L(u,\chi)} \,du
    = \log L(\tfrac{3}{2},\chi) - \log L(1,\chi) 
    = - \log L(1,\chi) + d\big(\log \zeta(\tfrac{3}{2})\big)O^{\star}( 1).
\end{equation*}

Lemma \ref{Lemma:LogChiEstimate} says that
\begin{equation*}\small
    \sum_p (\log{p}) \chi(p) p^{-u} e^{-p/x}
    = - \frac{L'(u,\chi)}{L(u,\chi)}
     + O^*\left(\frac{C_2 d x^{\frac{1}{2}+\delta-u}}{2\pi} \bigg(  \frac{C_4 \log N}{d} + C_5 \bigg)  + 0.76d\right).
\end{equation*}
Integrate this over $u\in [1,\frac{3}{2}]$ and observe that
\begin{align*}
    &\int_1^{3/2}\bigg[ \frac{C_2 d x^{\frac{1}{2}+\delta-u}}{2\pi} \bigg(  \frac{C_4 \log N}{d} + C_5 \bigg)  + 0.76d\bigg]\,du \\
    &\qquad\qquad= \frac{C_2 (\sqrt{x}-1)  (C_4 \log N+C_5 d)}{2 \pi x^{1-\delta}\log x}+ 0.38d \\
    &\qquad\qquad\leq \frac{C_2   (C_4 \log N+C_5 d)}{2 \pi x^{\frac{1}{2}-\delta}\log x}+0.38d \\
    &\qquad\qquad=\frac{ C_6 \log N + C_7 d}{x^{\frac{1}{2}-\delta} \log x} + 0.38d,
\end{align*}
to deduce that
\begin{align*}
    & \sum_p \chi(p) p^{-1}e^{-p/x} + d P(\tfrac{3}{2})O^{\star}(1) 
     = \int_1^{\frac{3}{2}}\sum_p (\log p) \chi(p) p^{-u} e^{-p/x}\,du\\
    &\qquad \quad= \log L(1,\chi) + d\big(\log \zeta(\tfrac{3}{2})\big)O^{\star}( 1) + O^*\left( \frac{ C_6 \log N + C_7 d}{x^{\frac{1}{2}-\delta} \log x} + 0.38d\right).
\end{align*}
Since $P(\frac{3}{2}) < 0.849567$
and $\log \zeta(\tfrac{3}{2}) < 0.96026$, we obtain the desired result.
\end{proof}

\subsection{Approximating exponentials}\label{Subsection:Exponential}
The sum that appears in the previous lemma involves the expression $e^{-p/x}$. Since $e^{-p/x} = 1 + O(1/x)$, one hopes to replace $e^{-p/x}$ with $1$ if $p/x$ is sufficiently small.
In what follows, let $\beta > \frac{1}{2}$\label{p:Beta} and define\label{p:xy}
\begin{equation}\label{eq:xy}
x = (\log N)^{\beta} \qquad \text{and} \qquad y = (\log N)^{1/2}.
\end{equation}

\begin{lemma}\label{Lemma:Nxy}
Let $L(s,\chi)$ be an entire Artin $L$-function that satisfies GRH, in which 
$\chi$ has degree $d$ and conductor $N$.  Then
\begin{equation*}
\Bigg|\sum_p \chi(p) p^{-1}e^{-p/x} - \sum_{p\leq y} \chi(p) p^{-1} \Bigg|  
\, <  \, d \log{4\beta} + d\, F (x,y),
\end{equation*}
where 
\begin{equation}\label{eq:F}
F(x,y) := \frac{y}{x}
+ m(x^2) + m( y)
+\frac{4}{xe^x}
\end{equation}
tends to zero as $N \to \infty$ for each fixed $\beta > \frac{1}{2}$.  Here $m(\cdot)$ denotes
the function \eqref{eq:mFunction}.
\end{lemma}

\begin{proof}
First observe that
\begin{align*}
\sum_p \chi(p) p^{-1}e^{-p/x} 
&= \sum_{p \leq y} \chi(p) p^{-1} + \underbrace{\sum_{p\leq y} \chi(p) p^{-1}(e^{-p/x}-1)}_{I_1}  \\
&\qquad\quad+ \underbrace{\sum_{y<p\leq x^2} \chi(p) p^{-1}e^{-p/x} }_{I_2}
+ \underbrace{\sum_{x^2<p} \chi(p) p^{-1}e^{-p/x} }_{I_3}.
\end{align*}
We estimate the summands $I_1$, $I_2$, and $I_3$ separately.
\medskip

\noindent\textsc{Bounding $I_1$.}
Since $p \leq y < x$, we may use $t=-p/x$ in the inequality 
\begin{equation}\label{eq:assemblyuno}
    |e^t - 1 | \leq |t| \quad \text{for $-1<t\leq 0$},
\end{equation}
which follows since the series $e^t-1 = \sum_{n=1}^{\infty} t^n/n!$ is alternating for such $t$.
Thus,
\begin{align*}
    |I_1| 
    &= \Big| \sum_{p\leq y} \chi(p) p^{-1}(e^{-p/x}-1) \Big| \\
    &\leq\, \sum_{p\leq y} |\chi(p)| \frac{|e^{-p/x}-1|}{p} \\
    &\leq\, d \sum_{p\leq y} \frac{1}{p} \cdot \frac{p}{x} && \text{(by \eqref{eq:assemblyuno})}\\
    &\leq \frac{d}{x} \pi(y) \\
    &\leq \frac{dy}{x} .
\end{align*}
Although this bound can be improved asymptotically, the improvement is negligible for the range of parameter values considered;
see Remark \ref{Remark:Rosser} below.
\medskip

\noindent\textsc{Bounding $I_2$.}
For $t \geq (\log 3)^{1/2} > 1.048$, recall 
that the bound \eqref{eq:SuperRosser} from Lemma \ref{Lemma:Mertens} applies.
Since $m(t)$ is decreasing,
\begin{align*}
    |I_2|
    &= \bigg| \sum_{y<p\leq x^2} \chi(p) p^{-1}e^{-p/x} \bigg| 
    \leq d \sum_{y<p\leq x^2} p^{-1}e^{-p/x} 
    \leq d \sum_{y<p\leq x^2} p^{-1} \\
    &\leq d \bigg( \sum_{p\leq x^2} p^{-1} - \sum_{p\leq y} p^{-1}\bigg)\\
    &\leq d \left[ \big(\log \log (x^2) + M + O^{\star}(m(x^2)) \big) - \big( \log \log y + M + O^{\star}(m(y))\big)  \right] \\
    &\leq d \bigg( \log \bigg( \frac{\log (x^2)}{\log y} \bigg) +  m(x^2) +m(y) \bigg)\\
    &= d  \bigg(  \log \bigg( \frac{2\beta\log \log N }{\frac{1}{2} \log \log N }\bigg)+  m(x^2) +m(y) \bigg)\\
    &= d  \big(  \log 4 \beta +  m(x^2) +m(y) \big).
\end{align*}

\noindent\textsc{Bounding $I_3$.}
Observe that
\begin{align*}
    | I_3 |
    &= \Big| \sum_{x^2<p} \chi(p) p^{-1}e^{-p/x}  \Big| 
    \leq d \sum_{x^2<p}  p^{-1}e^{-p/x} \\
    &\leq \frac{d}{x^2} \sum_{x^2<p} e^{-p/x} 
    \leq \frac{d}{x^2} \sum_{k = \lfloor x^2 \rfloor}^\infty (e^{-1/x})^k \\
    &=   \frac{d}{x^2} \frac{(e^{-1/x})^{\lfloor x^2 \rfloor }}{1 - e^{-1/x}} 
    \leq   \frac{d(e^{-1/x})^{x^2  }}{x^2(1 - e^{-1/x})} \\
    &\leq   \frac{de^{-x}}{x^2(1 - e^{-1/x})} 
    \leq \frac{4d}{x e^x}.
\end{align*}
The last inequality follows from the fact that $\frac{1}{4}|t| < |e^t - 1|$ for $0 < |t| < 1$.\footnote{Since $e^{-x}$ and $1 - e^{-1/x}$ tend to zero, computing their quotient leads to numerical issues.  Thus,
the final simplifying estimate is needed to ensure the numerical stability of later computations.}%
\medskip

Putting this all together yields the desired result.
\end{proof}

\begin{remark}\label{Remark:Rosser}
For $t > 1$, it is known that $\pi(t) < 1.25506 t / \log t$ \cite[Cor.~1, (3.5)]{Rosser}.
One can asymptotically improve the estimate in the proof of Lemma \ref{Lemma:Nxy} and get
\begin{equation*}
| I_1| \leq  \frac{d}{x} \pi(y) \leq \frac{1.25506y}{x \log y}d .
\end{equation*}
For the values of $\beta$ that arise in the numerical optimization step (Subsection \ref{Subsection:DukeComplete}), both estimates
yield essentially the same final result.  Therefore, we keep the original bound, which is a bit cleaner.
\end{remark}

\subsection{Completion of the proof of Theorem \ref{Theorem:Duke}}\label{Subsection:DukeComplete}
As defined in \eqref{eq:xy}, suppose that $x = (\log N)^{\beta}$ and $y = (\log N)^{1/2}$.
Using Lemmas \ref{Lemma:L1N} and \ref{Lemma:Nxy}, we find that 
\begin{equation}\label{eq:Gbound}
    \bigg|\log L(1,\chi) \,\,-  \!\!\!\!\!\!\sum_{p \leq (\log N)^{1/2}} \frac{\chi(p)}{p}  \bigg|
    \leq G(\beta,\delta,\eta;d,N)
\end{equation}
for all $\beta > \frac{1}{2} $, $\delta \in (0,\tau]$, and $\eta \in (1,\frac{3}{2})$,
where
\begin{align}
    G(\beta,\delta,\eta;d,N) &:=
    \underbrace{ 2.19d +  \frac{ C_6(\delta,\eta) \log N + C_7(\delta,\eta) d}{x^{\frac{1}{2}-\delta} \log x} }_{\text{from Lemma \ref{Lemma:L1N}}}
     + \underbrace{ d \log 4 \beta + d \, F(x,y) }_{\text{from Lemma \ref{Lemma:Nxy}}}\nonumber\\
     &=\underbrace{\frac{C_6(\delta,\eta) \log{N}}{x^{\frac{1}{2}-\delta} \log x}}_{A(\beta,\delta,\eta;N)}
     + \underbrace{\left(2.19 + \log 4 \beta + F(x,y)+  \frac{C_7(\delta,\eta) }{x^{\frac{1}{2}-\delta} \log x}\right)}_{B(\beta,\delta,\eta;N)} d \label{eq:G}
\end{align}
is an affine function of $d$.
It remains to optimize the constants $\eta$, $\delta$, and $\beta$.
Before we attempt this, let us make a few remarks about the qualitative nature of the expressions involved.
This informs our selection of the parameters below and justifies the particular search region that we eventually consider.

First observe that $F(x,y)$ contains the summand
\begin{equation*}
    m(y) =
    \frac{1}{\sqrt{y}} \left(\dfrac{3\log{y} + 4}{8\pi}\right)+ \frac{5}{y^2}
    \geq \frac{3 \log \log N }{16 \pi (\log N)^{1/4} }.
\end{equation*}
This error term comes from Lemma \ref{Lemma:Mertens}, which is an asymptotically sharp form of Mertens' second theorem
under the Riemann Hypothesis.  Thus, we cannot reasonably expect the error bound in \eqref{eq:Gbound}
to be better than $O((\log \log N) (\log N)^{-1/4})$.

Next note that $\log 4\beta$ tends to infinity slowly with $\beta$, and hence we can afford to 
make $\beta$ large if this reduces the other terms significantly.  In fact, the remark above suggests that we at least take
(recall the definition \eqref{eq:tau} of $\tau$ and that $\delta \leq \tau$)
\begin{equation*}\label{p:BetaBig}
    \beta \geq \frac{5}{2 -4 \tau} \approx 4.46003
\end{equation*}
so that 
\begin{equation*}
    \frac{ C_6 \log N}{x^{\frac{1}{2}-\delta} \log x}
    = \frac{ C_6 \log N}{ \beta (\log N)^{\beta(\frac{1}{2}-\delta)} \log \log N}
    = O\left( \frac{1}{ (\log N)^{\frac{1}{4}} \log \log N } \right)
\end{equation*}
is dominated by $m(y)$, which we already know cannot be much improved.

Aside from $m(y)$, a glance at \eqref{eq:F} shows that the other three summands that comprise $F(x,y)$ are of lower order in $N$. 
Moreover, increasing $\beta$ increases the rate at which these terms decay as $N\to\infty$.
Thus, we should take $\beta$ relatively large.

\begin{figure}
\centering
\includegraphics[width=0.6\textwidth]{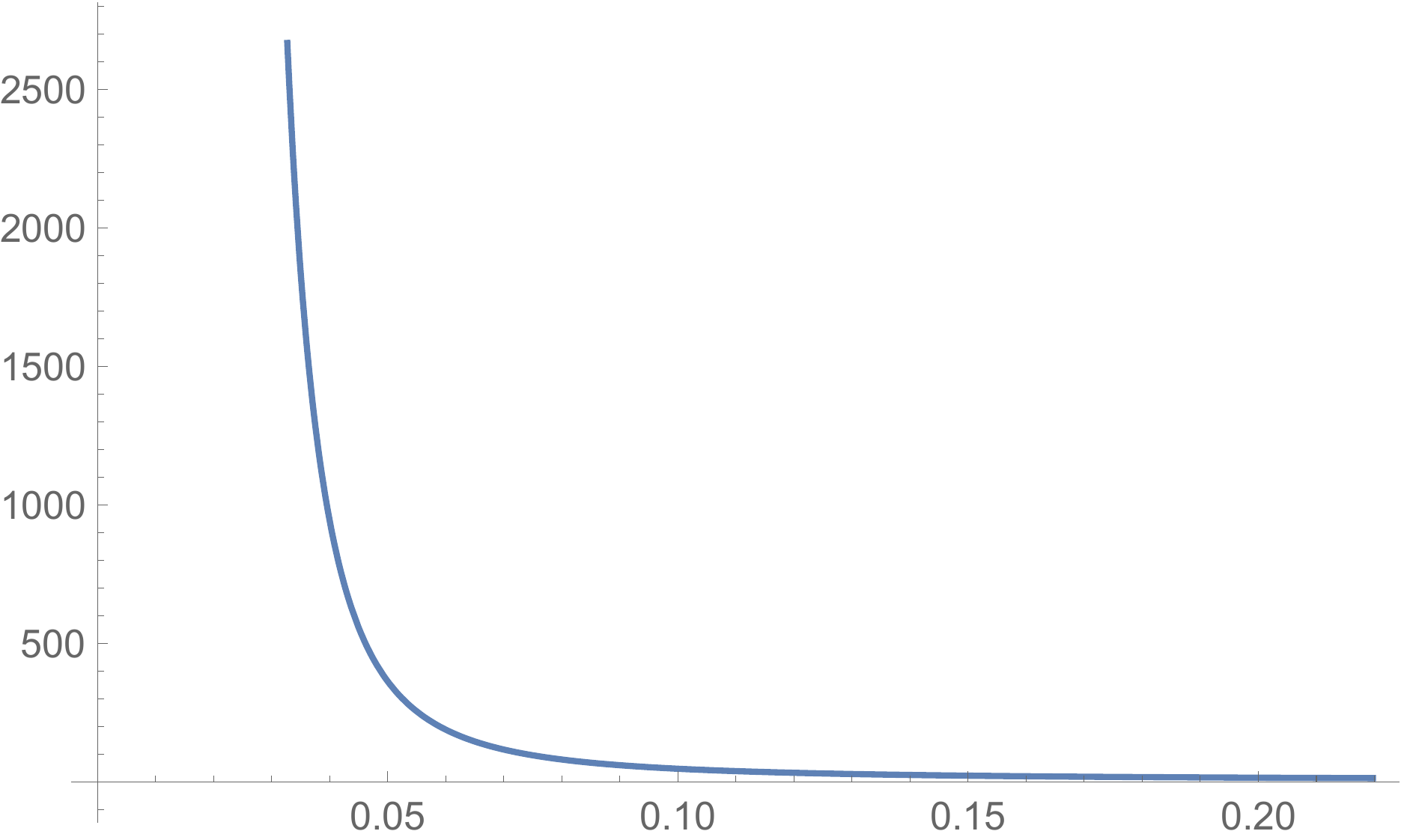}
\caption{Graph of $C_4(\delta)$ for $\delta \in (0,\tau]$.}
\label{Figure:C4}
\end{figure}

There is also the question about minimizing the constants $C_6(\delta,\eta)$ and $C_7(\delta,\eta)$,
which are defined in \eqref{eq:C6-7}.  This involves minimizing $C_2(\delta,\eta)$, $C_3(\eta)$, $C_4(\delta)$, and $C_5(\delta,\eta)$.
All of these tasks point in the same direction.
A look at the definitions \eqref{eq:C2-3} of $C_2(\delta,\eta)$ and \eqref{eq:C4} of $C_4(\delta)$
warns us to keep $\delta \in (0,\tau]$ away from zero.  In fact, $C_4(\delta)$ is minimized for $\delta \in (0,\tau]$ by selecting $\delta = \tau$; see Figure \ref{Figure:C4}.  
Thus, one expects that in the final optimization $\delta$ should be close to $\tau$.
The pole of the Riemann zeta function at $s=1$ and the definition \eqref{eq:C2-3} of $C_3(\eta)$ tell us that $\eta \in (1,\frac{3}{2}]$ must be kept away from $1$; this also ensures that $C_5(\delta,\eta)$ does not get out of hand.

For $\beta \geq 4.5$, all of the summands that comprise $G(\beta,\delta,\eta;d,N)$ are decreasing functions of $N$.  
Therefore, $G(\beta,\delta,\eta;d,N) \leq G(\beta,\delta,\eta;d,3)$ for $N \geq 3$ (recall that the smallest possible conductor for an entire Artin $L$-function is $3$; see the comments at the beginning of Subsection \ref{Subsection:DukeEstimate}).
Since $d \geq 1$ and
\begin{equation*}
    G(\beta,\delta,\eta;d,3) = A(\beta,\delta,\eta;3) + B(\beta,\delta,\eta;3) d \leq \big(A(\beta,\delta,\eta;3) + B(\beta,\delta,\eta;3)\big)d,
\end{equation*}
we minimize $A(\beta,\delta,\eta;3) + B(\beta,\delta,\eta;3)$ to obtain our final bound on \eqref{eq:Gbound}.

Numerical investigations with $\delta \approx \tau$ suggest that values of $\beta$ significantly larger than $200$ or smaller than $100$ hurt us;
see Figure \ref{Figure:Hurt}.  Thus, we settle on the domain
\begin{equation}\label{eq:SearchRegion}
(\beta,\delta,\eta) \in [ 4.5, 300] \times (0,\tau]\times (1,\tfrac{3}{2}]
\end{equation}
and rapidly obtain the (approximate) optimal point 
\begin{equation}\label{eq:Choice}
(\beta,\delta,\eta) = 
(155.648,\, 0.213503,\, 1.18818),
\end{equation}
which (approximately) minimizes the objective and yields
\begin{equation*}
\bigg|\log L(1,\chi) \,\,-  \!\!\!\!\!\!\sum_{p \leq (\log N)^{1/2}} \frac{\chi(p)}{p}  \bigg| 
 < 13.53d.
\end{equation*}
This completes the proof of Theorem \ref{Theorem:Duke}. \qed

\begin{figure}
\centering
\includegraphics[width=0.475\textwidth]{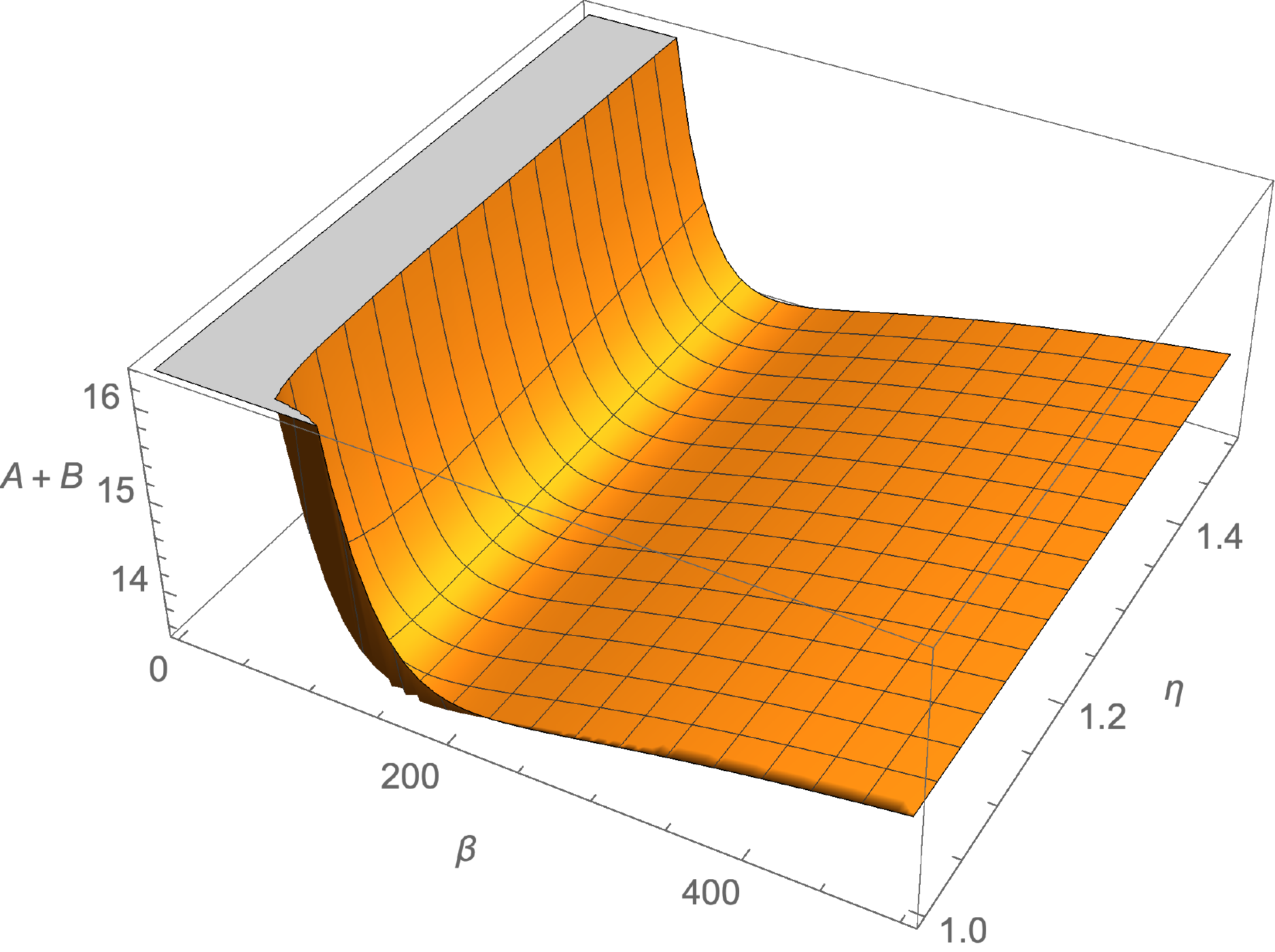}
\quad
\includegraphics[width=0.475\textwidth]{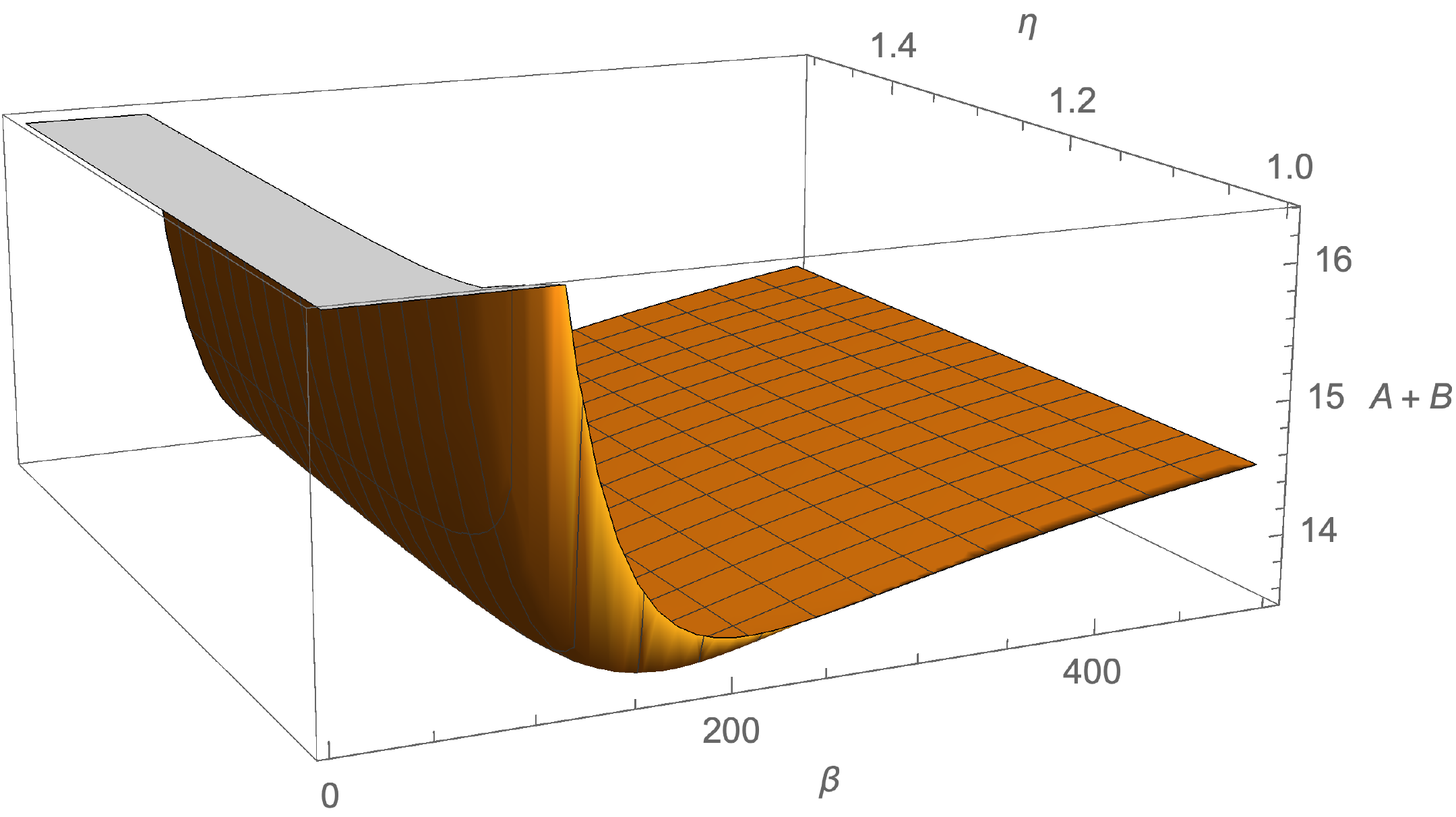}
\caption{Views of the surface $A(\beta,\tau,\eta;3) + B(\beta,\tau,\eta;3)$ (vertical axis)
for $\beta \in[4.5,500]$ and $\eta \in (1,\frac{3}{2}]$.}
\label{Figure:Hurt}
\end{figure}

\section{Proof of Corollary \ref{Corollary:Kappa}}\label{Section:ProofKappa}

The proof of Corollary \ref{Corollary:Kappa} relies upon the final stages of the proof of Theorem \ref{Theorem:Duke}.  
A fair amount of set up is needed before we begin and there are several cases that require special attention.

\subsection{Preliminaries}\label{Subsection:PreliminariesDukeThm}
Let $\K$ be a number field of degree $n_{\K} \geq 2$ over $\Q$ such that $\zeta_{\K}/\zeta$ is entire and let $\L/\Q$ denote the Galois closure of  the extension $\K/\Q$. 
Then 
\begin{equation}\label{eq:LZZ}\qquad
L(s,\chi) = \frac{\zeta_{\K}(s)}{\zeta(s)} = \sum_{k=1}^{\infty} \frac{\chi(k)}{k^s}
\quad \text{for $\Re s > 1$},
\end{equation}
in which $\chi$ is the character of a certain $(n_{\K}-1)$-dimensional representation of $\Gal(\L/\Q)$.
We claim that the conductor $N$ of $\chi$ satisfies $N = |\Delta_{\K}|$.
Indeed, the conductor of $\zeta_{\K}(s)$ is $|\Delta_{\K}|$ \cite[p.~125]{iwaniec2004analytic};
in particular, $\zeta(s) = \zeta_{\Q}(s)$ has conductor $1$.  Since $L(s,\chi) \zeta(s) = \zeta_{\K}(s)$,
the multiplicativity of the conductor \cite[p.~95]{iwaniec2004analytic} ensures that $N = |\Delta_{\K}|$.
Minkowski's bound implies that $N \geq 3$.

It is known that $L(s,\chi)$ is meromorphic on $\C$, and our assumption that $\zeta_{\K}/\zeta$ is entire ensures that $L(s,\chi)$ is entire.
Unconditionally, $L(s,\chi)$ has no zeros or poles on the vertical line $\Re s = 1$ \cite[Cor.~5.47]{iwaniec2004analytic}.
Since $\zeta(s)$ has a simple pole at $s=1$ with residue $1$, \eqref{eq:LZZ} implies that
\begin{equation*}
\kappa_{\K} = \underset{s=1}{\operatorname{Res}}\, \zeta_{\K}(s) = \lim_{s\to 1} (s-1)\zeta(s) L(s,\chi) = \lim_{s\to 1} L(s,\chi) = L(1,\chi).
\end{equation*}
In particular, this ensures that $L(1,\chi) > 0$ since $\zeta_{\K}(s)/\zeta(s) > 0$ for $s>1$ and because $\zeta_{\K}(s)$
has a simple pole at $s=1$.

Dirichlet convolution shows that the coefficients in \eqref{eq:LZZ} satisfy
\begin{equation*}
\chi(k) = \sum_{i|k} a_k \mu\bigg(\frac{k}{i}\bigg),
\end{equation*}
in which $a_k$ denotes the number of ideals in $\K$ of norm $k$ and $\mu$ is the M\"obius function.  Therefore,
\begin{equation}\label{eq:Chip}
\chi(p) = a_p\mu(1) + a_1 \mu(p) = a_p - 1 \in [-1,n_{\K}-1],
\end{equation}
and, in particular, $|\chi(k)|\leq n_{\K} - 1$.
Indeed, if $p$ is inert in $\K$, then $N(p) = p^{n_{\K}} \neq p$ and 
$\chi(p)=-1$.  The other extreme  $\chi(p) = n_{\K}-1$ occurs if $p$ splits totally in $\K$.

\subsection{Upper and lower bounds}\label{Subsection:GeneralSetup}
Since $N \geq 3$, Lemma \ref{Lemma:Mertens} provides
\begin{equation*}
\sum_{p \leq (\log N)^{\frac{1}{2}}} \frac{1}{p}
\,\,\leq\,\,
\log \log \log N + \log \tfrac{1}{2}+ M + m(\sqrt{ \log N } ) .
\end{equation*}
It follows from the previous inequality, \eqref{eq:Gbound}, and \eqref{eq:Chip} that
\begin{align*}
    \log L(1,\chi)
    &\leq  \sum_{p \leq (\log N)^{\frac{1}{2}}} \frac{\chi(p)}{p} + G(\beta,\delta,\eta;d,N) 
    \leq d \!\!\!\!   \sum_{p \leq (\log N)^{\frac{1}{2}}} \frac{1}{p} + G(\beta,\delta,\eta;d,N)  \\
    &= d  \log \log \log N + H(\beta,\delta,\eta;d,N),
\end{align*}
in which
\begin{equation*}
H(\beta,\delta,\eta;d,N) = d\big(\log \tfrac{1}{2} + M  + m( \sqrt{\log N}) \big) + G(\beta,\delta,\eta;d,N).
\end{equation*}
Similarly,
\begin{align*}
\log L(1,\chi)
&\geq\,\, -\!\!\!\! \sum_{p \leq (\log N)^{\frac{1}{2}}} \!\!\!\!  p^{-1}   - G(\beta,\delta,\eta;d,N) \\
&\geq- \log \log\log N - \log \tfrac{1}{2}  -  M - m(\sqrt{\log N} ) - G(\beta,\delta,\eta;d,N) \\
&= - \log \log \log N - K(\beta,\delta,\eta;d,N),
\end{align*}
in which
\begin{equation*}
K(\beta,\delta,\eta;d,N) = \log \tfrac{1}{2} + M  + m( \sqrt{\log N})  
+ G(\beta,\delta,\eta;d,N).
\end{equation*}
Since $L(1,\chi)>0$, we exponentiate the inequalities above,
perform the substitutions $d = n_{\K}-1$ and $N = | \Delta_{\K} |$,
and obtain
\begin{equation}\label{eq:Dedekind}
\frac{1}{e^{K(\beta,\delta,\eta;d,N)} \log \log | \Delta_{\K} | }
\,\leq\, \kappa_{\K} \,\leq\, (\log \log |\Delta_{\K}|)^{n_{\K}-1} e^{H(\beta,\delta,\eta;d,N)}.
\end{equation}
We now optimize the parameters $\beta$, $\delta$, and $\eta$ 
to suit our particular needs.

\subsection{General case}
We can optimize $H$ and $K$ in the same manner
that we treated $G$ in Subsection \ref{Subsection:DukeComplete}.
The functions $H(\beta,\delta,\eta;d,N)$ and $K(\beta,\delta,\eta;d,N)$ are affine in $d$
and decreasing in $N$ when the other variables are fixed in the search region
\eqref{eq:SearchRegion}.
The optimal choices of $(\beta,\delta,\eta)$ in both cases equal \eqref{eq:Choice} (to the precision displayed).
This is to be expected since the variables $\beta$, $\delta$, and $\eta$ only appear in the definitions of $H$ and $K$ as arguments of the summand 
$G(\beta,\delta,\eta;d,N)$.
In the end, we obtain
\begin{equation*}
\frac{1}{e^{17.81 (n_{\K}-1)} \log \log | \Delta_{\K}| }
\leq \kappa_{\K} \leq  (e^{17.81}\log \log | \Delta_{\K}|)^{n_{\K}-1} .
\end{equation*}

\subsection{Fields of small degree}\label{ssec:fieldsofsmalldegree}
For $2 \leq n_{\K} \leq 6$, we can make some small improvements
(still assuming that the function $\zeta_{\K}/ \zeta$ is entire).
Let $N_0:= N_0(n)$ denote the minimum absolute discriminant
of a number field of degree $n$.
For $n=2,3,4,5,6$, these are $3, 23, 117, 1609, 9747$,
respectively; see Table \ref{Table:NumberFields}.

\begin{table}
\begin{equation*}
\begin{array}{c|cclc}
n & N_0 & \Delta_{\K} & \text{Defining polynomial for $\K$} & \text{LMFBD Entry} \\
\hline
2 & 3 & -3 & x^2-x+1 & \href{https://www.lmfdb.org/NumberField/2.0.3.1}{2.0.3.1}\\[2pt]
3 & 23 & 23 & x^3 - x^2 + 1  & \href{https://www.lmfdb.org/NumberField/3.1.23.1}{3.1.23.1}\\[2pt]
4 & 117 & 117 & x^4 - x^3 - x^2 + x + 1  & \href{https://www.lmfdb.org/NumberField/4.0.117.1}{4.0.117.1} \\[2pt]
5 & 1609 & 1609 & x^5 - x^3 - x^2 + x + 1 & \href{https://www.lmfdb.org/NumberField/5.1.1609.1}{5.1.1609.1}\\[2pt]
6 & 9747 & -9747 & x^6 - x^5 + x^4 - 2x^3 + 4x^2 - 3x + 1 & \href{https://www.lmfdb.org/NumberField/6.0.9747.1}{6.0.9747.1}
\end{array}
\end{equation*}
\caption{Number fields $\K$ with minimal absolute discriminant $N_0 = | \Delta_{\K}|$ for a given degree $n=2,3,\ldots,6$,
along with defining polynomials and corresponding entry in the \href{https://www.lmfdb.org/}{The $L$-functions and Modular Forms Database}.
}
\label{Table:NumberFields}
\end{table}

Since $d = n_{\K}-1$ and because $H(\beta,\delta,\eta;d,N)$ and $K(\beta,\delta,\eta;d,N)$ are decreasing functions of $N$
for each fixed $(\beta,\delta,\eta)$ in \eqref{eq:SearchRegion}, we can select
$(\beta,\delta,\eta)$ to minimize $H$ and $K$, respectively.  It turns out that approximately the same
triple minimizes both $H$ and $K$, so we quote a single triple for each pair $(d,N_0)$ in
Table \ref{Table:HK}.  This leads to the following results quoted in Corollary \ref{Corollary:Kappa}.

This completes the proof of Corollary \ref{Corollary:Kappa}. \qed

\begin{table}
\begin{equation}
\begin{array}{c|ccccc}
n_{\K} & d & N_0 & (\beta,\delta,\eta) & H(\beta,\delta,\eta,d, N_0) & K(\beta,\delta,\eta,d, N_0)\\
\hline
2 & 1 & 3 & (155.648, 0.213503, 1.18818) &17.809 & 17.809\\[3pt]
3 & 2 & 23 & (13.0627, 0.210516, 1.16757) & 18.8667 & 17.5328\\[3pt]
4 & 3 & 117& (9.57219, 0.210398, 1.16682) & 24.0981 & 22.5199\\[3pt]
5 & 4 & 1609& (7.5451, 0.208941, 1.15761) & 28.1733 & 26.9298\\[3pt]
6 & 5 & 9747 & (6.80012, 0.208989, 1.15791) & 33.3541 &32.2334
\end{array}
\end{equation}
\caption{Minimum absolute discriminant $N_0$ 
for a number field $\K$ of degree $n_{\K}$, a good choice of 
$(\beta,\delta,\eta)$, and the corresponding values $H(\beta,\delta,\eta,d, N_0)$ and $K(\beta,\delta,\eta,d, N_0)$.}
\label{Table:HK}
\end{table}

\appendix
\section{Constants and Functions}\label{Section:Constants}
There are many constants and auxiliary functions, 
many of which are interdependent, that appear in the proof of Theorem \ref{Theorem:Duke}.
For the sake of convenience,
we summarize below those expressions that arise in a non-local context
(that is, not isolated to a single environment or subsection).

\begin{itemize}
\item $M = 0.261497212847642783755\dots$ (Meissel--Mertens constant) \hfill (p.~\pageref{p:M})
\item $m(x) = \dfrac{1}{\sqrt{x}} \left(\dfrac{3\log{x} + 4}{8\pi}\right)+ \dfrac{5}{x^2}$ \hfill (p.~\pageref{eq:mFunction})

\item $C_1 = \dfrac{\zeta( 3/2)}{\sqrt{2\pi}} = 1.0421869788690765546\ldots$ \hfill (p.~\pageref{p:C1})

\item $\eta \in (1, \frac{3}{2}]$ \hfill (p.~\pageref{p:eta})
\item $\delta \in (0, \frac{1}{2})$ (later restricted to $\delta \in (0,\tau]$ in Lemma \ref{Lemma:GammaLog}) \hfill (p.~\pageref{p:eta})

\item $C_2(\delta,\eta) = \dfrac{2\eta-1}{2\delta^2}$ \hfill (p.~\pageref{eq:C2-3})
\item $C_3(\eta) =  \left(C_1 \dfrac{\zeta(\eta)}{\zeta(2\eta)} \right)^2$ \hfill (p.~\pageref{eq:C2-3})

\item $\tau \approx 0.2197330687867739$ \hfill (p.~\pageref{eq:tau})

\item $C_4(\delta) =\dfrac{ 4 \sqrt{2} \, \delta ^{\delta -\frac{1}{2}} e^{\frac{1}{6 \delta }} }{\sqrt{\pi}(1-\delta )}$ \hfill (p.~\pageref{eq:C4})

\item $C_5(\delta, \eta) = C_4\big(\log  C_3 + \frac{\pi}{2}\big)$ \hfill (p.~\pageref{eq:C5})

\item $u \in [1, \frac{3}{2}]$ \hfill (p.~\pageref{p:u})

\item $C_6(\delta,\eta)= \dfrac{C_2(\delta,\eta) C_4(\delta)}{2\pi}$ \hfill (p.~\pageref{eq:C6-7})

\item $C_7(\delta,\eta) =  \dfrac{C_2(\delta,\eta) C_5(\delta,\eta)}{2 \pi}$ \hfill (p.~\pageref{eq:C6-7})

\item $\beta > \frac{1}{2}$ (later restricted to $\beta > 4.5 > 4.46003=\frac{5}{2 -4 \tau}$ on p.~\pageref{p:BetaBig}) \hfill (p.~\pageref{p:Beta})

\item $x = (\log N)^{\beta}$ \hfill (p.~\pageref{p:xy})

\item $y = (\log N)^{1/2}$ \hfill (p.~\pageref{p:xy})

\item $F(x,y) = \dfrac{(e-1)y}{x}+ m(x^2) + m( y)+\dfrac{4}{xe^x}$ \hfill (p.~\pageref{eq:F})

\item $A(\beta,\delta,\eta,N) =\dfrac{C_6(\delta,\eta) \log{N}}{x^{\frac{1}{2}-\delta} \log x}$ \hfill (p.~\pageref{eq:G})

\item $B(\beta,\delta,\eta,N) =2.19 + \log 4 \beta + F(x,y)+  \dfrac{C_7(\delta,\eta) }{x^{\frac{1}{2}-\delta} \log x}$ \hfill (p.~\pageref{eq:G})

\item $G(\beta,\delta,\eta;d,N) =A(\beta,\delta,\eta,N) + B(\beta,\delta,\eta,N) d$ \hfill (p.~\pageref{eq:G})

\end{itemize}

\bibliographystyle{amsplain}
\bibliography{EEALFDSSTDZR}

\end{document}